\documentclass[a4paper,11pt]{amsart}
\usepackage{amssymb,graphicx}
\newtheorem{cond}{Condition}
\newtheorem{lemma}{Lemma}
\newtheorem{prop}{Proposition}
\newtheorem{thm}{Theorem}
\newtheorem{cor}{Corollary}
\newtheorem*{result}{Theorem}
\theoremstyle{definition}
\newtheorem{defn}{Definition}
\theoremstyle{remark}
\newtheorem{rem}{Remark}

\newtheorem{ex}{Example}
\newtheorem*{exs}{Examples}
\newcounter{numl}
\newcommand{\labelnuml}{\textup{(\roman{numl})}}
\newenvironment{numlist}{\begin{list}{\labelnuml}%
{\usecounter{numl}\setlength{\leftmargin}{0pt}%
\setlength{\itemindent}{2\parindent}%
\setlength{\itemsep}{\smallskipamount}\def
\makelabel ##1{\hss \llap {\upshape ##1}}}}{\end{list}}
\newenvironment{bulletlist}{\begin{list}{\labelitemi}%
{\setlength{\leftmargin}{\parindent}\def
\makelabel ##1{\hss \llap {\upshape ##1}}}}{\end{list}}

\DeclareSymbolFont{script}{U}{eus}{m}{n}
\DeclareSymbolFontAlphabet{\mathscr}{script}
\DeclareMathSymbol{\Wedge}{0}{script}{"5E}
\DeclareMathAlphabet{\mathrmsl}{OT1}{cmr}{m}{sl}
\newcommand{\R}{{\mathbb R}}
\newcommand{\C}{{\mathbb C}}
\newcommand{\Z}{{\mathbb Z}}
\newcommand{\N}{{\mathbb N}}
\newcommand{\Q}{{\mathbb Q}}
\newcommand{\T}{{\mathbb T}}
\newcommand{\cH}{{\mathcal H}}
\newcommand{\cC}{{\mathcal C}}
\newcommand{\cF}{{\mathcal F}}
\newcommand{\cL}{{\mathcal L}}
\newcommand{\cO}{{\mathcal O}}

\newcommand{\sub}{\subseteq}

\newcommand{\trace}{\mathop{\mathrm{tr}}\nolimits}
\newcommand{\g}{\mathfrak g}
\newcommand{\tor}{{\mathfrak t}}
\newcommand{\torh}{{\mathfrak h}}
\newcommand{\Fa}{F}
\newcommand{\Lab}{{\mathbf L}}
\renewcommand{\d}{{\mathrmsl d}}
\newcommand{\s}[1]{{\mathit s}_{#1}}
\newcommand{\Hess}{\mathop{\mathrm{Hess}}}
\newcommand{\eps}{\varepsilon}

\newcommand{\transp}{^\top}
\newcommand{\ip}[1]{\langle #1 \rangle}
\newcommand{\spn}[1]{\mathopen< #1\mathclose>}
\newcommand{\spns}{\mathrm{span}}
\newcommand{\del}{\partial}
\newcommand{\ang}{\boldsymbol t}
\newcommand{\taumap}{\boldsymbol\tau}
\newcommand{\bx}{\boldsymbol x}
\newcommand{\by}{\boldsymbol y}
\newcommand{\bz}{\boldsymbol z}
\newcommand{\Proj}{\mathrm P}
\newcommand{\eaf}{\zeta}
\newcommand{\w}[1]{\sigma^{(#1)}}
\newcommand{\p}[1]{p^{(#1)}}
\newcommand{\pl}[1]{\lambda^{(#1)}}
\newcommand{\pr}[1]{\rho^{(#1)}}
\newcommand{\me}{\lambda}
\newcommand{\al}{\alpha}
\newcommand{\be}{\beta}
\newcommand{\ga}{\gamma}
\newcommand{\alvec}{\boldsymbol\alpha}
\newcommand{\bevec}{\boldsymbol\beta}
\newcommand{\gavec}{\boldsymbol\gamma}
\newcommand{\thvec}{\boldsymbol\theta}
\newcommand{\bp}{\boldsymbol\pi}
\newcommand{\bs}{\boldsymbol s}
\newcommand{\bt}{\boldsymbol t}
\newcommand{\dd}{\eps}
\newcommand{\ad}{\eta}
\newcommand{\wt}{w}
\newcommand{\1}{0}
\newcommand{\2}{\infty}

\newcommand{\fs}{\varphi}
\begin{document}
\strut\vspace{-1cm}
\title[Ambitoric geometry II]
{Ambitoric geometry II:\\ Extremal toric surfaces and Einstein $4$-orbifolds}
\author[V. Apostolov]{Vestislav Apostolov}
\address{Vestislav Apostolov \\ D{\'e}partement de Math{\'e}matiques\\
UQAM\\ C.P. 8888 \\ Succursale Centre-ville \\ Montr{\'e}al (Qu{\'e}bec) \\
H3C 3P8 \\ Canada}
\email{apostolov.vestislav@uqam.ca}
\author[D.M.J. Calderbank]{David M. J. Calderbank}
\address{David M. J. Calderbank \\ Department of Mathematical Sciences\\
University of Bath\\ Bath BA2 7AY\\ UK}
\email{D.M.J.Calderbank@bath.ac.uk}
\author[P. Gauduchon]{Paul Gauduchon}
\address{Paul Gauduchon \\ Centre de Math\'ematiques\\
{\'E}cole Polytechnique \\ UMR 7640 du CNRS\\ 91128 Palaiseau \\ France}
\email{pg@math.polytechnique.fr}
\date{February 2013}
\begin{abstract} We provide an explicit resolution of the existence problem
for extremal K\"ahler metrics on toric $4$-orbifolds $M$ with second Betti
number $b_2(M)=2$.  More precisely we show that $M$ admits such a metric if
and only if its rational Delzant polytope (which is a labelled quadrilateral)
is K-polystable in the relative, toric sense (as studied by S.~Donaldson,
E. Legendre, G. Sz\'ekelyhidi et al.). Furthermore, in this case, the extremal
K\"ahler metric is \emph{ambitoric}, i.e., compatible with a conformally
equivalent, oppositely oriented toric K\"ahler metric, which turns out to be
extremal as well. These results provide a computational test for the
K-stability of labelled quadrilaterals.

Extremal ambitoric structures were classified locally in Part I of this work,
but herein we only use the straightforward fact that explicit K\"ahler metrics
obtained there are extremal, and the identification of Bach-flat (conformally
Einstein) examples among them. Using our global results, the latter yield
countably infinite families of compact toric Bach-flat K\"ahler orbifolds,
including examples which are globally conformally Einstein, and examples which
are conformal to complete smooth Einstein metrics on an open subset, thus
extending the work of many authors.
\end{abstract}
\maketitle
\vspace{-0.4cm}

\section*{Introduction}

This paper concerns the explicit construction of extremal K\"ahler metrics on
compact $4$-orbifolds, including K\"ahler metrics which are conformally
Einstein (either globally or on the complement of real hypersurface).  The
examples we construct are toric with second Betti number two, i.e., their
rational Delzant polytope (which is the image of the momentum map of the
$2$-torus action~\cite{Delzant,LT})) is a quadrilateral. More precisely, we
use extremal \emph{ambitoric} metrics, which we classified locally in Part I
of this work, to resolve completely the existence problem in the quadrilateral
case.

There are several narratives to which this paper may be viewed as a
contribution. A general theme is the interplay between the abstract existence
theory for a geometric PDE, and the construction of explicit solutions
associated to special geometric structures.  Extremal K\"ahler metrics were
introduced by E.~Calabi~\cite{Cal1,Cal2} to address the problem of finding
canonical K\"ahler metrics with K\"ahler form in a given cohomology class
$\Omega$ on a compact complex manifold.  The ${\mathrm L}_2$ norm of the scalar
curvature yields a functional on $\Omega$, and its critical points are the
extremal metrics. They are thus natural generalizations of constant curvature
metrics on Riemann surfaces; in general, the Euler--Lagrange equation asserts
that a K\"ahler metric is extremal if its scalar curvature is hamiltonian for
a Killing vector field. As a geometric PDE, this is quasilinear of fourth
order, and no general methods are currently available.

Nevertheless, considerable progress on the existence theory has been made,
following the seminal work of Calabi~\cite{Cal0} on the non-positive
K\"ahler--Einstein case and the resolution of his famous conjecture by
T.~Aubin~\cite{Aubin} and S-T.~Yau~\cite{Yau0}.  Conjectures going back to
Yau~\cite{Yau1}, G.~Tian~\cite{Tian} and S.~Donaldson~\cite{Donaldson1} state
that the obstruction to the existence of an extremal K\"ahler metric in the
class $\Omega=2\pi c_1(\cL)$ of a polarized complex manifold $(M,\cL)$ should
be a purely algebro-geometric ``stability condition'' on the pair $(M,\cL)$,
and these conjectures may be extended to orbifolds~\cite{RT2}. Defining a
precise notion of stability is part of the problem, one candidate being
``K-(\emph{poly})\emph{stability}''~\cite{Tian,Donaldson1}: the necessity of
K-polystability has been proven for constant scalar curvature
metrics~\cite{Donaldson5,CT,Stoppa,Mabuchi}, and a version of K-polystability
relative to a maximal torus of the automorphism group of $(M,\cL)$, developed
by G.~Sz\'ekelyhidi~\cite{Sz1,Sz2}, is necessary for the existence of an
extremal K\"ahler metric of non-constant scalar curvature~\cite{Stoppa-Sz}.

A major difficulty with the theory is that in practice it is not only
difficult to determine whether a given polarized variety admits an extremal
K\"ahler metric---it is also difficult to verify a proposed stability
condition.  Consequently classes of complex manifolds or orbifolds for which
extremality and stability are more tractable play an important role. These
examples come in two main flavours: \emph{ruled} and \emph{toric}. Using a
construction due to Calabi~\cite{Cal1}, ruled surfaces and other projective
line bundles provide a setting for many explicit extremal K\"ahler
metrics~\cite{Christina1,Hwang-Singer,Sz1,ACGT}. We refer to these as metrics
\emph{of Calabi type}; they admit a hamiltonian $2$-form of order
one~\cite{ACG,ACG2}.  The extremality equations reduce to ODEs with explicit
polynomial solutions, and stability amounts to a positivity condition on the
solution~\cite{Sz1,ACGT2}.  For toric varieties, in contrast, the extremality
equations only reduce to a nonlinear fourth order PDE in the momenta; explicit
solutions are hard to find, but the existence theory is
well-developed~\cite{Donaldson1,Donaldson2} and there is a well-understood
notion of ``relative K-polystability with respect to toric degenerations''
which is widely believed to be equivalent to
existence~\cite{Donaldson1,Sz2,ZZ1,ZZ2}.  Explicit examples are largely
limited to orthotoric $2m$-orbifolds, which admit a hamiltonian $2$-form of
order $m$ and have a convex $m$-cube (or degeneration) for their rational
Delzant polytope~\cite{ACG,ACGT,Eveline}.

In dimension four, examples and theory come together to provide a moderately
complete picture. Extremal K\"ahler surfaces of Calabi type are locally toric
(as the base is a constant curvature Riemann surface) and there are specific
results for toric surfaces: for example, K-polystability implies uniform
K-polystability~\cite{Donaldson1,Sz2} and for constant scalar metrics,
existence is equivalent to K-polystability~\cite{Donaldson3,Donaldson6}.

Our paper is closely related to work of E.~Legendre~\cite{Eveline}, who
investigated systematically the extent to which explicit methods resolve the
existence problem when the rational Delzant polytope is a convex
quadrilateral.  Her solution highlighted the role of the \emph{extremal affine
  function} $\eaf$ on the rational Delzant polytope, a combinatorial invariant
which pulls back to the scalar curvature in the extremal case. Her main
results show that hamiltonian $2$-form methods suffice only for ``equipoised''
quadrilaterals, for which $\eaf$ has equal values at the midpoints of the
diagonals. A key ingredient in Legendre's work is the observation that $\eaf$
is linear in the inverse lengths of the normals. Using this, she resolved the
existence problem for the codimension one family of equipoised quadrilaterals
using orthotoric, Calabi type or product metrics.

The theory of hamiltonian $2$-forms in four dimensions~\cite{ACG} implies that
these toric metrics are in fact \emph{ambitoric}, i.e., toric with respect to
a pair of oppositely oriented but conformally equivalent K\"ahler metrics. The
local classification of ambitoric structures~\cite{ACG1} implies that the
``regular'' examples (i.e., neither a product nor of Calabi type) are
determined by a quadratic polynomial $q$ and two functions $A,B$ of one
variable. Regular ambitoric structures reduce to orthotoric metrics precisely
when $q$ has vanishing discriminant. The extremality conditions for regular
ambitoric structures can be explicitly solved with $A,B$ given by quartic
polynomials~\cite{ACG1}, and this generalization suffices to remove the
equipoisedness constraint introduced by Legendre.

To prove this, we use, in addition to ambitoric geometry, two further
ingredients.  The first is an analysis of rational Delzant quadrilaterals
building on~\cite{Eveline2}. We compute the extremal affine function $\eaf$
and establish a notion of ``temperateness'' for polystable quadrilaterals
which implies $\eaf$ is positive at the midpoints of the diagonals.

The second is the concept of a ``factorization structure'', which makes
precise the separation of variables technique that underpins explicit
solutions of geometric PDEs on toric $4$-orbifolds.  One can hope such an
approach will work in $2m$-dimensions when the rational Delzant polytope is a
convex $m$-cube (or degeneration), with the $2m$ facets providing boundary
conditions for the $m$ functions of one variable determining the solution. In
particular, by the uniqueness of toric extremal K\"ahler metrics~\cite{Guan},
we might \emph{expect} a rational Delzant polytope to select an essentially
unique adapted factorization structure for the solution. This is indeed what
happens for $m=2$.

The fruit of this analysis is Theorem~\ref{thm:stability}, which establishes,
for quadrilaterals, Donaldson's conjecture~\cite{Donaldson1} that the
existence of an extremal K\"ahler metric is equivalent to relative
K-polystability with respect to toric degenerations. Indeed, we show that
temperate quadrilaterals admit a factorization structure which relates the
polystability condition directly to the positivity of the quartics $A$ and $B$
appearing in the expression for the extremal ambitoric metrics. This
explicitly computable criterion yields new examples both of extremal toric
$4$-orbifolds, and of (unstable) toric $4$-orbifolds admitting no extremal
K\"ahler metric.

In our discussion of examples, we return to another motivation for ambitoric
geometry: K\"ahler metrics which are conformally Einstein. Since the work of
D.~Page~\cite{Page} and E.~Calabi~\cite{Cal1}, such metrics have been an
important source of examples, with contributions by
L.~B\'erard-Bergery~\cite{Berard-Bergery}, R.~Bryant~\cite{Bryant},
A.~Derdzinski~\cite{De}, G.~Maschler~\cite{DM}, and
C.~LeBrun~\cite{LeBrun1,LeBrun2a,LeBrun2b} among others. In part I of this
work~\cite{ACG1}, we classified locally $4$-dimensional Einstein metrics with
degenerate half-Weyl tensors using Bach-flat ambitoric structures (which are
extremal and locally conformally Einstein). Here we show that Bach-flat
ambitoric $4$-orbifolds are abundant, and include examples which are globally
conformally Einstein, as well as examples with an open set where the K\"ahler
metric is conformal to a smooth, complete (conformally compact) Einstein
metric on a covering. These extend in particular the examples of
R.~Bryant~\cite{Bryant}

\medbreak

The organization is as follows. In section~\ref{s:toric-orb} we review the
theory of compact toric K\"ahler
orbifolds~\cite{Abreu1,Delzant,Donaldson1,Guillemin,LT}, but adopting an
affine invariant viewpoint. We begin our analysis of quadrilaterals in
section~\ref{s:sq} where affine invariance provides an effective tool to
compute, for example, the extremal vector field, without extensive calculus.
By considering the affine structure as a variable, we similarly use projective
invariance to simplify our later discussion of factorization structures. This
approach is closely related to $5$-dimensional contact, CR and sasakian
geometry, cf.~\cite{Eveline2}, which we discuss in Appendix~\ref{s:CR}. The
main results are established in
sections~\ref{s:orbi-amb}--\ref{s:ext-orbi-amb} which concern the
compactification of ambitoric metrics in general, in terms of
\emph{factorization structures}, and extremal ambitoric metrics in particular
in terms of \emph{adapted factorization structures}. Examples, including the
new Einstein metrics, are given in section~\ref{s:examples}. In
Appendix~\ref{s:semistablity}, we study the K-semistability surface and show
that any quadrilateral which is not a parallelogram can be made K-unstable for
suitable choices of labels.

\medbreak

The first author was supported by an NSERC Discovery Grant and the second
(partially) by an Advanced Research Fellowship.  The authors are grateful to
Liana David and the Centro Georgi, Pisa, for an opportunity to meet in 2006,
and to the Simons Institute, Stony Brook, for a workshop invitation in
2011. They also thank Miguel Abreu, Hugues Auvray, Olivier Biquard, Claude
LeBrun, \'Eveline Legendre, Gabor Sz\'ekelyhidi and Christina T\o
nnessen-Friedman for helpful discussions and comments.

\section{Toric orbifolds, K\"ahler metrics and polystability}
\label{s:toric-orb}

We review the theory of toric K\"ahler $2m$-orbifolds $M$, primarily adopting
the symplectic point of view, as
in~\cite{Abreu1,Delzant,Guillemin,Guillemin-book,LT}. We denote the $m$-torus
acting on $M$ by $\T=\tor/2\pi\Lambda$, where $\tor$ is its (abelian) Lie
algebra, and $\Lambda$ its lattice of circle subgroups. Our applications have
$m=2$ and a geometry which may not be compatible with the lattice or origin.
We therefore use basis-independent and affine-invariant language.

\subsection{Toric symplectic orbifolds}\label{s:toric-symp}

Let $\torh$ be an $(m+1)$-dimensional real vector space, and $\iota\colon\R\to
\torh$ a $1$-dimensional subspace with quotient $\tor$.  Dually, $\torh^*$ has
$\tor^*$ as a subspace with quotient $\iota\transp\colon \torh^*\to\R$.  The
inverse image of $1\in\R$ under $\iota\transp$ is an affine subspace $\Xi$ of
$\torh^*$, modelled on $\tor^*$. The space of affine functions $f$ on $\Xi$ is
canonically isomorphic to $\torh$, where the constant functions are
$\iota(c)$, $c\in \R$, and the projection of $f$ to $\tor$, viewed as a linear
form on $\tor^*$, is its derivative $\d f$ (at every point of $\Xi$).

\begin{defn} Let $L_1,\ldots L_n$ be affine functions on $\Xi$ such
that the convex polytope
\begin{equation*}
\Delta:=\{\xi \in\Xi: L_j(\xi)\geq 0,\ j=1,\ldots n\}
\end{equation*}
is compact and nonempty.\footnote{We implicitly assume each $L_j$ vanishes
somewhere on $\Delta$, otherwise it may be discarded.} Then
$(\Delta,L_1,\ldots L_n)$ is a \emph{rational Delzant polytope} in $\Xi$ iff
\begin{numlist}
\item $\forall\,j\in\{1,\ldots n\}$ the \emph{normals} $u_j:=\d L_j\in\tor$
  belong to the lattice $\Lambda\subset\tor$ and
\item $\forall\xi\in\Delta$, $N_\xi:=\{u_j\in \tor: L_j(\xi)=0\}$ is linearly
independent in $\tor$.
\end{numlist}
\end{defn}
The term ``rational'' refers to the fact that the normals $u_j$ span an
$m$-dimensional vector space over $\Q$. If the \emph{affine normals} $L_j$
span an $(m+1)$-dimensional vector space over $\Q$, we say the polytope is
\emph{strongly rational}. The \emph{faces} $\Fa$ of $\Delta$ are intersections
of the \emph{facets} (codimension one faces) $\Fa_j=\Delta\cap\{\xi\in\Xi:
L_j(\xi)=0\}$ which have inward normals $u_j$. A rational Delzant polytope is
\emph{simple} or \emph{$m$-valent}: $m$ facets and $m$ edges meet at each
vertex. The primitive inward normals, which are uniquely determined by
$\Delta$ and $\Lambda$, have the form $u_j/m_j$ for some positive integer
labelling $m_j$ of the facets $\Fa_j$, so rational Delzant polytopes are also
called \emph{labelled polytopes}~\cite{LT}. It is convenient to encode the
labelling in the row vector $\Lab=(L_1,\ldots L_n)\in\mathrm{Hom}(\R^n,\torh)
\cong\torh\otimes \R^{n*}$ of affine normals, and denote the normals by
$\d\Lab=(u_1,\ldots u_n) \in \mathrm{Hom}(\R^n,\tor)$.

Compact toric symplectic $2m$-orbifolds are classified (up to equivariant
symplecto\-morphism) by rational Delzant polytopes (up to lattice preserving
affine equivalences)~\cite{Delzant,LT}. In one direction, given a toric
symplectic orbifold $M$, $\Delta$ is the image of the natural momentum map
$\mu\colon M\to\torh^*$, where $\torh$ is the vector space of hamiltonian
generators of $\T$, and $(\spns_\R u_j\cap \Lambda)/\spns_\Z u_j\cong
\Z/m_j\Z$ is the local uniformizing group of every point in
$\mu^{-1}(\Fa_j^0)$. (For any face $\Fa$, we denote by $\Fa^0$ its interior.)
Conversely, $(\Delta,\Lab)$ determines $(M,\omega)$ as a symplectic quotient
of $\C^n$ by an $(n-m)$-dimensional subgroup $G$ of the standard $n$-torus
$(S^1)^n=\R^n/2\pi\Z^n$: $G$ is the kernel of the map $(S^1)^n\to
\T=\tor/2\pi\Lambda$ induced by the natural map $\d\Lab\colon \R^n\to \tor$
(with kernel the Lie algebra $\g$ of $G$); the composite of $\Lab\transp\colon
\torh^*\to \R^{n*}$ with the transpose $\R^{n*}\to \g^*$ of the inclusion
therefore vanishes on $\tor^*$ and hence induces a map $\lambda\colon\R=
\torh^*/\tor^*\to\g^*$---the momentum level for the symplectic quotient of
$\C^n$ by $G$ is then $\lambda(1)$.

\begin{rem} Affine functions $L_j$ defining a rational Delzant polytope
$\Delta\subset\Xi$ do so with respect to \emph{any} lattice containing the
normals $u_j$. There is clearly a smallest such lattice
$\Lambda:=\spns_\Z\{u_j:j=1,\ldots n\}$, and any other such lattice $\Lambda'$
contains $\Lambda$ as a sublattice (of finite index).  The torus
$\T'=\tor/2\pi\Lambda'$ is the quotient of $\T=\tor/2\pi\Lambda$ by a finite
abelian group $\Gamma\cong \Lambda'/\Lambda$, and the corresponding toric
symplectic orbifolds $M$ and $M'$ (under the tori $\T$ and $\T'$) are related
by a regular orbifold covering~\cite{Thurston}: $M'=M/\Gamma$. In fact $M$ is
a \emph{simply connected orbifold} in the sense of W. Thurston~\cite{Thurston}
and is the universal orbifold cover of $M'$~\cite{LT}. We therefore say that
the rational Delzant polytope is \emph{simply connected}. Simply connected
rational Delzant polytopes are entirely determined by the affine normals
$\Lab$.
\end{rem}

\subsection{Toric K\"ahler orbifolds}\label{s:kahler-compact}

We next consider K\"ahler metrics compatible with a toric symplectic
structure.  On the union $M^0:=\mu^{-1}(\Delta^0)$ of the generic orbits, such
metrics have an explicit general expression due to
V.~Guillemin~\cite{Guillemin,Guillemin-book}. In this description, the
momentum map $\mu\colon M^0\to \Xi$ is supplemented by angular coordinates
$\ang\colon M\to \tor/2\pi\Lambda$ such that the kernel of $\d\ang$ is
orthogonal to the torus orbits. These action-angle coordinates $(\mu,\ang)$
identify each tangent space to $M^0$ with $\tor\oplus \tor^*$, and the
symplectic form is $\omega=\ip{\d\mu\wedge \d\ang}$, where $\ip{\cdot}$
denotes contraction of $\tor$ and $\tor^*$.  Hence invariant
$\omega$-compatible K\"ahler metrics on $M^0$ have the form
\begin{equation}\label{toricmetric}
g=\ip{\d\mu, {\mathbf G} , \d\mu}+ \ip{ \d\ang,{\mathbf H}, \d\ang},
\end{equation}
where ${\mathbf G}$ is a positive definite $S^2\tor$-valued function of $\mu$,
${\mathbf H}$ is its pointwise inverse in $S^2\tor^*$ (at each point,
${\mathbf G}$ and ${\mathbf H}$ define mutually inverse linear maps $\tor^*
\to\tor$ and $\tor\to\tor^*$) and $\ip{\cdot,\cdot,\cdot}$ denotes the
pointwise contraction $\tor^* \times S^2\tor \times \tor^* \to \R$ or the
dual contraction.  The corresponding almost complex structure is defined by
\begin{equation}\label{toricJ}
J \d\ang = -\ip{ {\mathbf G}, \d\mu},
\end{equation}
and $J$ is integrable if and only if ${\mathbf G}$ is the Hessian of a
function~\cite{Guillemin}.

Necessary and sufficient conditions for $\mathbf H$ to come from a globally
defined metric on $M$ are obtained in \cite{Abreu1,ACGT,Donaldson2}.  Here we
use the first order boundary conditions given in~\cite[\S1]{ACGT}.  In order
to state them, we denote by ${\tor}_\Fa \sub {\tor}$ (for any face $\Fa\sub
\Delta$) the vector subspace spanned by the inward normals $u_j \in {\tor}$ to
facets containing $\Fa$. Thus the tangent plane to points in $\Fa^0$ is the
annihilator ${\tor}^0_\Fa\cong({\tor}/{\tor}_\Fa)^*$ of ${\tor}_\Fa$ in
$\tor^*$.

\begin{prop}\label{p:toric-ccs} Let $(M,\omega)$ be a compact toric symplectic
$2m$-manifold or orbifold with natural momentum map $\mu \colon M\to \Delta
\subset \Xi\subset \torh^*$, and ${\mathbf H}$ be a positive definite
$S^2\tor^*$-valued function on $\Delta^0$.  Then ${\mathbf H}$ defines a
$\T$-invariant, $\omega$-compatible almost K\"ahler metric $g$ via
\eqref{toricmetric} if and only if it satisfies the following
conditions\textup:
\begin{bulletlist}
\item \textup{[smoothness]} ${\mathbf H}$ is the restriction to $\Delta^0$ of a
smooth $S^2\tor^*$-valued function on $\Delta$\textup;
\item \textup{[boundary values]} for any point $\xi$ on the facet $\Fa_j
\subset \Delta$ with inward normal $u_j$,
\begin{equation}\label{toricboundary}
{\mathbf H}_{\xi}(u_j, \cdot) =0\qquad {and}\qquad (d{\mathbf
H})_{\xi}(u_j,u_j) = 2 u_j,
\end{equation}
where the differential $d{\mathbf H}$ is viewed as a smooth $S^2\tor^*\otimes
{\tor}$-valued function on $\Delta$\textup;
\item \textup{[positivity]} for any point $\xi$ in interior of a face $\Fa
\sub \Delta$, ${\mathbf H}_{\xi}(\cdot, \cdot)$ is positive definite when
viewed as a smooth function with values in $S^2({\tor}/{\tor}_\Fa)^*$.
\end{bulletlist}
\end{prop}

\subsection{The extremal affine function and K-polystability}\label{s:eaf-stab}

Let $(M,J,g,\omega)$ be a compact K\"ahler orbifold invariant under the action
of a maximal torus $G$ in the reduced automorphism group $H_0(M,J)$ of
$(M,J)$.  (By a result of Calabi~\cite{Cal1}, any extremal K\"ahler metric is
invariant under such a $G$.) Following \cite{FM}, the \emph{extremal potential}
is the ${\mathrm L}_2$-projection of the scalar curvature $\s{g}$ onto the
space of Killing potentials (with respect to $\omega$) of elements of the Lie
algebra $\g$ and the \emph{extremal vector field} is its symplectic
gradient. A. Futaki and T. Mabuchi~\cite{FM} show that the extremal vector
field is independent of the choice of a $G$-invariant K\"ahler metric within
the given K\"ahler class $[\omega]$ on $(M,J)$.  Since, the extremal vector
field is central, $G$ can also be taken to be a maximal compact
subgroup. Furthermore, by adopting the symplectic
viewpoint~\cite{Fujiki,Donaldson,Lejmi}, the extremal potential becomes a
natural deformation invariant of the complex structure, for fixed
$(M,\omega,G)$.

For toric symplectic orbifolds $(M,\omega,\T)$, the extremal potential is an
element $\eaf$ of $\torh$, called the \emph{extremal affine
function}~\cite{ACGT2,Guan,Lejmi}, and is defined (in the notation of
section~\ref{s:toric-symp}) by the following vector equation in $\torh^*$:
\begin{equation*}
\int_{\xi\in\Delta} \ip{\xi,\eaf}\xi \, \d\me = \int_{\xi\in\Delta}
\s{g}(\xi)\xi \,\d\me = 2\int_{\xi\in \del\Delta} \xi \,\d\nu
\end{equation*}
where $\d\me$ is a (constant) volume form on $\Xi$, the $(m-1)$-form $d\nu$
satisfies $u_j\wedge\d\nu=-\d\me$ on the facet $\Fa_j$ of $\del\Delta$ with
normal $u_j$, and $\s{g}$ is the scalar curvature of a compatible K\"ahler
metric viewed as a function on $\Delta$. The combinatorial boundary integral
for the first moment of $\s{g}$ is an application of the Abreu
formula~\cite{Abreu0}
\begin{equation}\label{abreu}
\s{g}=-\mathrm{div}\, \delta\mathbf{H}:=
-\sum_{r,s}\frac{\partial^2 H_{rs}}{\partial\xi_r\partial\xi_s}
\end{equation}
for the scalar curvature of the compatible metric defined by $\mathbf H$,
together with the divergence theorem; the latter calculation uses only the
boundary conditions of Proposition~\ref{p:toric-ccs}, and not the positive
definiteness of $\mathbf H$ on the faces of $\Delta$, nor the fact that
$\mathbf H$ is the inverse hessian of a symplectic potential. We deduce that
if $\mathbf H$ satisfies the boundary conditions and $-\mathrm{div}\,
\delta\mathbf{H}$ is an affine function, then this is the extremal affine
function; such an $\mathbf H$ is called a \emph{formal extremal solution}.

The extremal affine function is important not only as the scalar curvature of
a compatible extremal K\"ahler metric, but also because it may be used to
define a \emph{relative Futaki invariant} and hence a combinatorial
K-polystability criterion~\cite{Donaldson1,Donaldson3,Sz2,ZZ1}.

\begin{defn}\label{K-stability-def}
The \emph{relative Futaki invariant} $\cF_{\Delta,\Lab}$ of a compact toric
symplectic $2m$-orbifold $(M,\omega,\T)$ with rational Delzant polytope
$(\Delta,\Lab)$ is defined by
\begin{equation}\label{futaki}
\cF_{\Delta,\Lab}(f):=\int_{\xi\in \del\Delta} f(\xi) \d\nu
-\frac12 \int_{\xi\in\Delta} \ip{\xi,\eaf} f(\xi) \d\me
\end{equation}
for any continuous function $f$ on $\Delta$. Note that $\cF_{\Delta,\Lab}$
vanishes on affine functions $f$.
\end{defn}
Note that if $\mathbf H$ is a formal extremal solution, we may substitute
$\eaf=-\mathrm{div}\, \delta\mathbf{H}$ in this formula and integrate by
parts to obtain
\begin{equation}\label{Hfutaki}
\cF_{\Delta,\Lab}(f) = \frac12 \int_{\xi\in\Delta} \trace({\mathbf H} \, \Hess
f) \d\me.
\end{equation}

Let $\mathcal{PL}(\Delta)$ be the space of continuous piecewise-linear (PL)
convex functions $f$ on $\Delta$ (thus $f$ is the maximum of a finite
collection of affine linear functions). Although~\eqref{Hfutaki} involves two
derivatives of $f$, it may be used in a distributional sense to compute
$\cF_{\Delta,\Lab}(f)$ for $f\in \mathcal{PL}(\Delta)$. In particular
(cf.~\cite{Eveline}) let $f$ be a simple convex PL function with crease on the
line $\{\xi\in\Xi: \ip{\xi,u_f}=0 \}$ (with $u_f$ normalized to be the change
in $\d f$ along the line) and let $S_f$ be the intersection of this line with
$\Delta$. Then
\begin{equation}\label{Sfutaki}
\cF_{\Delta,\Lab}(f) = \int_{S_f} {\mathbf H}(u_f,u_f) \d\nu_f,
\end{equation}
where $\nu_f$ is the positive measure on $S_f$ such that $u_f \wedge \d\nu_f =
\d\me$.

\begin{defn} $(M,\omega,\T)$ is said to be (\emph{analytically, relatively})
\emph{K-polystable} (\emph{with respect to toric degenerations}) provided that
$\cF_{\Delta,\Lab}(f) \ge 0$ for all $f\in\mathcal{PL}(\Delta)$, with equality
iff $f$ is an affine function.
\end{defn}
The main conjecture of~\cite{Donaldson1} is that a compact toric orbifold
$(M,\omega,\T)$ admits a compatible extremal K\"ahler metric if and only if is
K-polystable in this toric sense. The forward implication has been established
by Zhou and Zhu~\cite{ZZ2}.  Conversely, in~\cite{Donaldson3}, Donaldson shows
that for polygons with zero extremal vector field, this toric K-polystability
criterion implies existence of a CSC metric.  The general extremal case
remains open, which motivates its study in the ambitoric context.

\section{Simplices and quadrilaterals}\label{s:sq}

Rational Delzant polytopes may be considered from a projective viewpoint, not
just an affine one. To fix notation, for a real vector space $V$, we denote by
$\Proj(V)=V^\times/\R^\times$ the nonzero vectors in $V$ up to scale:
$\Proj(V)$ is isomorphic to the set of $1$-dimensional subspaces of $V$, where
the equivalence class $[v]\in\Proj(V)$ of a nonzero vector $v\in V$ is mapped
to $\spn{v}\sub V$, its span.

Given a rational Delzant polytope $(\Delta,\Lab)$ we may identify
$\Delta\subset\Xi\subset\torh^*$ with its image in $\Proj(\torh^*)$ and also
with the convex cone in $\torh^*$ of its nonnegative multiples. Dually, the
space $\Delta^*\sub\torh$ of affine functions which are nonnegative on
$\Delta$ is a convex cone,\footnote{In fact $\Delta^*$ is a \emph{strictly}
  convex cone: it contains no nontrivial linear subspace.} given by the
nonnegative linear combinations of the affine normals $L_1,\ldots L_n$, which
we may identify with its image in $\Proj(\torh)$.  The incarnations of
$\Delta$ determine one another uniquely, but depend upon the choice of $\Xi$,
or equivalently, the inclusion $\iota\colon\R\to\torh$ or the \emph{affine
  structure} $\iota(1)\in\torh$; note that $\iota(1)$ is in the interior of
$\Delta^*$.

\subsection{Rational Delzant simplices}\label{s:simplex}

The case of $m$-simplices is well understood, but we summarize it briefly,
both as a warm-up, and because we shall use the case of triangles ($m=2$) as a
limiting case of quadrilaterals. All simplices are affine equivalent, so
simply connected rational Delzant simplices are parametrized by the choices
of scale for the normals. Concretely let $\Delta\subset\Xi$ be the $m$-simplex
on which $\ell_j(\xi)\geq 0$ for affine functions $\ell_0,\ell_1,\ldots
\ell_m$ with $\ell_0+\ell_1+\cdots+\ell_m=1$ on $\Xi$, so that each $\ell_j=1$
at the vertex $v_j$ opposite to the facet $\Fa_j$ on which it vanishes. Then
for any $r_0,r_1,\ldots r_m\in \R^+$ with rational ratios, affine normals
$L_j:=\ell_j/r_j$ define a rational Delzant simplex $(\Delta,\Lab)$ with
$\Lab=(L_0,L_1,\ldots L_m)$ and $\sum_{j=0}^m r_j L_j=1$.

The corresponding symplectic orbifolds are \emph{weighted projective spaces}:
the vector $(r_0,r_1,\ldots r_m)$ spans the kernel of the map $\d
\Lab\colon\R^{m+1}\to \tor$ sending $(x_0,x_1,\ldots x_m)$ to $\sum_{j=0}^m
x_j u_j$ (where $u_j=\d L_j$), and some multiple $(\wt_0,\wt_1,\ldots \wt_m)$
of $(r_0,r_1,\ldots r_m)$ is a list of positive integers with no common
multiple; the rational Delzant construction therefore yields the weighted
projective space $\C P^m_{\wt_0,\ldots \wt_m}$ as the symplectic quotient of
$\C^{m+1}$ by the diagonal action of $S^1$ with weights $\wt_0,\ldots
\wt_m$. Each weighted projective space has a unique K\"ahler class (up to
scale) containing a unique extremal K\"ahler metric (up to homothety and
biholomorphism), and this metric is Bochner-flat~\cite{Bryant}. In four
dimensions, a K\"ahler metric is Bochner-flat iff it is selfdual ($W_-=0$)
which is equivalent to the existence of many (local) opposite complex
structures, although none of these are globally defined on a weighted
projective plane.

The toric geometry of weighted projective spaces (and their quotients) has
been worked out in detail by M. Abreu~\cite{Abreu1} in specific coordinates.
Here we give an affine invariant derivation, as we shall use similar ideas to
simplify computation in the more complicated case of quadrilaterals.

\begin{lemma}\label{l:simplicial-integral} Let $\Delta$ be an $m$-simplex in
$\Xi$ as above, let $\me$ be a translation invariant measure on $\Xi$, and
$A_1,A_2\colon \Xi\to \R$ be affine functions whose values on the vertices
$v_0,v_1,\ldots v_m$ of $\Delta$ are given by $a_1,a_2\in \R^{m+1}$. Then
\begin{equation*}
\int_\Delta A_1 A_2 \,\d\me = B(a_1,a_2) \me(\Delta)
\end{equation*}
where $B$ is the symmetric bilinear form on $\R^{m+1}$ with $B_{jk}=
\frac{1+\delta_{jk}}{(m+1)(m+2)}$.
\end{lemma}
\begin{proof} Since all $m$-simplices are affine equivalent, and any affine
function is uniquely determined by its values on the $m+1$ vertices of a
simplex, the integral must have the given form for some symmetric bilinear
form $B$.  The entries of $B$ must be permutation invariant, so
$B_{jk}=a+b\delta_{jk}$.  Substituting $A_1=A_2=1$, we obtain
$(m+1)^2a+(m+1)b=1$. If $A_1=A_2=\ell_0$ (i.e., equal to $1$ at $v_0$ and $0$
at $v_1,\ldots v_m$), then we observe that for $0\leq x\leq 1$, $\me(\{p\in
\Delta:\ell_0(p)^2\geq x\})= \me(\{p\in \Delta:1-\ell_0(p)\leq 1-\sqrt
x\})=\me\bigl((1-\sqrt{x})\Delta\bigr) =(1-\sqrt x)^m\me(\Delta)$, and
integrating over $x$, the integral evaluates to
$2\me(\Delta)/(m+1)(m+2)$. Thus $(m+1)(m+2)(a+b)=2$ and $a=b=1/(m+1)(m+2)$.
\end{proof}

The measure $\nu$, with $u_j\wedge \d\nu=-\d\me$ on the facet $\Fa_j$ where
$\ell_j=0$, satisfies $\nu(\Fa_j)=m r_j\me(\Delta)$, since $L_j=1/r_j$ at the
opposite vertex $v_j$. Consequently, for any affine function $A$,
\begin{equation*}
\int_{\Fa_j} A \,\d\nu = r_j \me(\Delta)\sum_{k\neq j} A(v_k).
\end{equation*}
Since $(B^{-1})_{jk} = (m+1)\bigl((m+2)\delta_{jk}-1\bigr)$, the extremal
affine function $\eaf$ of $(\Delta,\Lab)$ is $\sum_{j=0}^m \eaf_j r_j$ where
\begin{equation*}
\eaf_j = \sum_{k=0}^m \tfrac12(m+1)(2-(m+2)\delta_{jk})\ell_k.
\end{equation*}
Note that $\eaf$ is linear in the parameters $r_j$: this is the reason for
using such an inverse scale to parametrize the normals. For $m=2$, $\eaf/3 =
(-\ell_0+\ell_1+\ell_2)r_0 +(\ell_0-\ell_1+\ell_2)r_1
+(\ell_0+\ell_1-\ell_2)r_2$, which is positive on the interior of the medial
triangle (with vertices at the midpoints of the edges of $\Delta$).  This
positivity has an analogue for convex quadrilaterals, to which we now turn.

\subsection{Rational Delzant quadrilaterals}\label{s:quad}

Quadrilaterals are not all affine equivalent, but they are projectively
equivalent since the vertices (or the projective normals) give four points in
general position in $\Proj(\torh^*)$ (or $\Proj(\torh)$).  Consequently,
quadrilaterals can be parametrized conveniently by varying the affine
structure, an approach adopted by E.~Legendre in~\cite{Eveline,Eveline2} and
closely related to 5-dimensional toric sasakian geometry (see
Appendix~\ref{s:CR}). Following Legendre, let $\Delta=\{[w,x,y]\in
\Proj(\R^{3*}):w\geq |x|, w\geq |y|\}$ be the quadrilateral with vertices
$[1,\pm1,\pm1]$. In the affine subspace $\{(w,x,y)\in \R^{3*}:w=1\}$ defined
by $(1,0,0)\in \R^{3}$, $\Delta$ is a square. More generally, affine subspaces
meeting $\Delta$ in a compact convex quadrilateral are parametrized by
vectors in the interior of the dual cone $\Delta^*$, spanned by $(1,\pm1,0)$
and $(1,0,\pm1)$. Any such vector is a positive multiple of
$\bigl(1,\frac12(\dd+\ad),\frac12(\dd-\ad)\bigr)$ for some $\dd,\ad\in \R$
with $|\dd|<1$ and $|\ad|<1$. The corresponding affine subspace is
$\{(w,x,y):2w+(\dd+\ad)x+(\dd-\ad)y=2\}$ and the vertices of $\Delta$ in this
subspace are
\begin{equation*}
v_{\1\1}=\frac{(1,-1,-1)}{1-\dd},\quad v_{\1\2}=\frac{(1,-1,1)}{1-\ad}, \quad
v_{\2\1}=\frac{(1,1,-1)}{1+\ad},\quad v_{\2\2}=\frac{(1,1,1)}{1+\dd}
\end{equation*}
with $v_{\1\1}$ opposite to $v_{\2\2}$ and $v_{\1\2}$ opposite to $v_{\2\1}$.
An affine function $A$ is uniquely determined by its values at the vertices,
but these values are constrained by the equality of two expressions for
(twice) the value of $A$ at the intersection of the diagonals:
\begin{equation}\label{affine-constraint}
(1-\dd)A(v_{\1\1})+(1+\dd)A(v_{\2\2})=(1-\ad)A(v_{\1\2})+(1+\ad)A(v_{\2\1}).
\end{equation}

The affine functions obtained by restricting $w+x$, $w-x$, $w+y$ and $w-y$
to this affine subspace will be denoted $\ell'_{\al,\1}$, $\ell'_{\al,\2}$,
$\ell'_{\be,\1}$ and $\ell'_{\be,\2}$ respectively. They clearly satisfy
$\ell'_{\al,\1}+\ell'_{\al,\2}=\ell'_{\be,\1}+\ell'_{\be,\2}$. We also set
\begin{align*}
\ell_{\al,\1}&=\tfrac14(1+\dd)(1+\ad)\ell'_{\al,\1},&
\ell_{\al,\2}&=\tfrac14(1-\dd)(1-\ad)\ell'_{\al,\2},\\
\ell_{\be,\1}&=\tfrac14(1+\dd)(1-\ad)\ell'_{\be,\1},&
\ell_{\be,\2}&=\tfrac14(1-\dd)(1+\ad)\ell'_{\be,\2},
\end{align*}
which satisfy $\ell_{\al,\1}+\ell_{\al,\2}+\ell_{\be,\1}+\ell_{\be,\2}=1$, and
whose nonzero values on vertices are
\begin{align*}
\ell_{\al,\1}(v_{\2\1})=\ell_{\be,\1}(v_{\1\2})&=\tfrac12(1+\dd),&
\ell_{\al,\1}(v_{\2\2})=\ell_{\be,\2}(v_{\1\1})&=\tfrac12(1+\ad)\\
\ell_{\al,\2}(v_{\1\1})=\ell_{\be,\1}(v_{\2\2})&=\tfrac12(1-\ad)&
\ell_{\al,\2}(v_{\1\2})=\ell_{\be,\2}(v_{\2\1})&=\tfrac12(1-\dd).
\end{align*}

These affine functions provide an affine invariant description of a family of
quadrilaterals $\Delta_{\dd,\ad}$ with $(\dd,\ad)\in (-1,1)\times (-1,1)$, and
we can drop the $(w,x,y)$ coordinate system. The parameters $\dd$, $\ad$ can
be interpreted geometrically in terms of the diagonals, which bisect each
other in the ratios $1-\dd:1+\dd$ and $1-\ad:1+\ad$. Inverse scales
$r_{\al,k},r_{\be,k}$ (where $k\in\{\1,\2\}$) for affine normals
$L_{\al,k}:=\ell_{\al,k}/r_{\al,k}$ and $L_{\be,k}:=\ell_{\be,k}/r_{\be,k}$
then define a rational Delzant quadrilateral $(\Delta,\Lab)$ in
$\Proj(\torh^*)$ (with an ordering of its vertices and the affine normals
$\Lab$ indexed $L_{\be,\1}$, $L_{\al,\1}$, $L_{\be,\2}$, $L_{\al,\2}$)
provided that the normals $u_{\al,k}=\d L_{\al,k}$ and $u_{\be,k}=\d
L_{\be,k}$ span a lattice in $\tor$.

The following diagram shows the projection of the quadrilateral onto the
$(x,y)$ plane. This normal form is orthodiagonal and so its Varignon
parallelogram (whose vertices are the midpoints of the sides of the
quadrilateral) is a rectangle. Also shown are the midpoints $v_\dd, v_\ad$ of
the diagonals and the centroid $v_0$.

\begin{center} 
\includegraphics{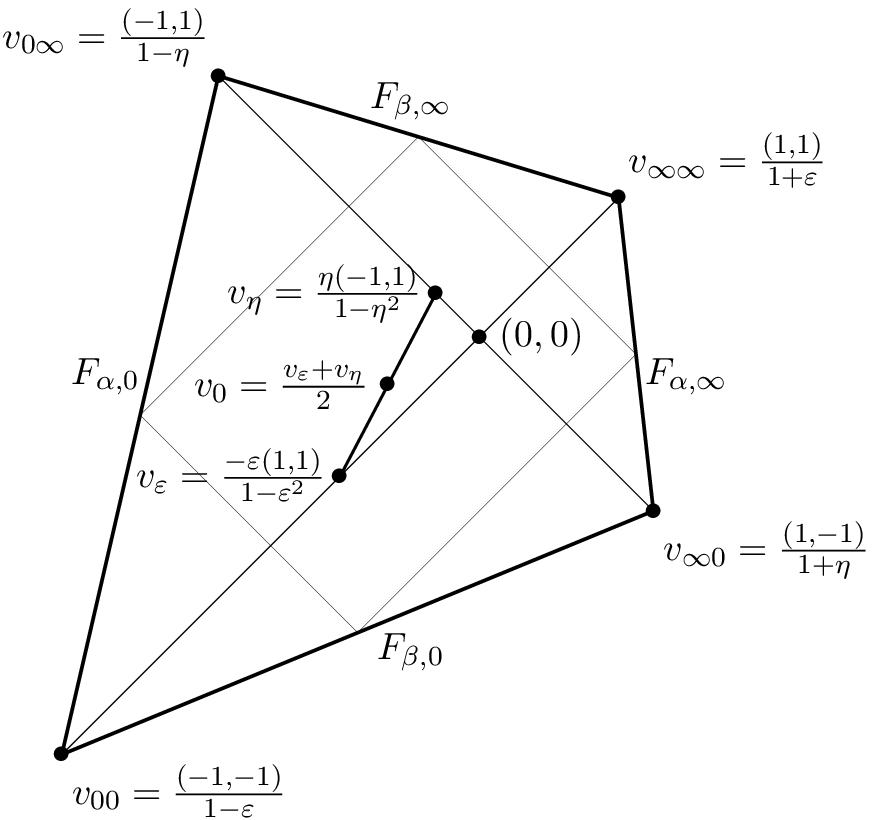}

Figure 1. A rational Delzant quadrilateral with its diagonals and Newton line.
\end{center}

When $\dd=\ad=0$, $\Delta$ is a parallelogram (these are all affine
equivalent); the associated simply connected symplectic $4$-orbifolds are
products of weighted projective lines (including $\C P^1\times \C P^1$ when
$r_{\al,\1}=r_{\al,\2}$ and $r_{\be,\1}=r_{\be,\2}$).  If $\ad=\pm\dd$,
$\Delta$ is a trapezium (two parallel sides); the associated simply connected
symplectic $4$-orbifolds are orbifold weighted projective line bundles over a
weighted projective line (which include the smooth Hirzebruch surfaces
$\Proj(\cO\oplus\cO(k))\to \C P^1$).

The extremal affine function $\eaf$ may be written $\eaf=c(\dd,\ad)
\sum_{k=\1,\2}(\eaf_{\al,k}r_{\al,k}+\eaf_{\be,k}r_{\be,k})$ where the
normalization constant $c(\dd,\ad)$ will be chosen shortly. By symmetry, it
suffices to compute $\eaf_{\al,\1}$, noting that its integral over $\Delta$
against any affine function $A$ depends only on the value of $A$ at the
midpoint of the edge $v_{\1\1}v_{\1\2}$.
\begin{lemma}\label{l:eaf-eqns} $\eaf_{\al,\1}$ satisfies the following
equations:
\begin{align}\label{eaf-eqn1}
\eaf_{\al,\1}(v_{\1\2})+\eaf_{\al,\1}(v_{\2\1})
+(1-\dd)\eaf_{\al,\1}(v_{\2\2})&=0\\
\eaf_{\al,\1}(v_{\1\1})+\eaf_{\al,\1}(v_{\2\2})
+(1-\ad)\eaf_{\al,\1}(v_{\2\1})&=0.
\label{eaf-eqn2}
\end{align}
\end{lemma}
\begin{proof} Let $A_\dd$ be an affine function which is constant on the
diagonal $v_{\2\1}v_{\1\2}$ and vanishes at the midpoint of $v_{\1\1}v_{\1\2}$.
Up to scale, we may take $A_\dd(v_{\2\1})=A_\dd(v_{\1\2})=1+\dd$,
$A_\dd(v_{\1\1})=-(1+\dd)$, and $A_\dd(v_{\2\2})=3-\dd$ (which
verifies~\eqref{affine-constraint}).  The integral of $\eaf_{\al,\1}A_\dd$
over $\Delta$ (which vanishes) may be computed using
Lemma~\ref{l:simplicial-integral} by splitting $\Delta$ into two triangles
along the diagonal $v_{\2\1}v_{\1\2}$. Up to a positive constant, the result is
\begin{equation*}
(1-\dd)(3+\dd)\bigl(\eaf_{\al,\1}(v_{\1\2})+\eaf_{\al,\1}(v_{\2\1})\bigr)
+4(1-\dd) \eaf_{\al,\1}(v_{\2\2})+(1+\dd)^2\eaf_{\al,\1}(v_{\1\2})
\end{equation*}
which readily yields~\eqref{eaf-eqn1}; \eqref{eaf-eqn2} follows similarly,
using the diagonal $v_{\1\1}v_{\2\2}$.
\end{proof}

It follows that $\eaf_{\al,\1}$ is a constant multiple of the affine function
\begin{multline*}
\bigl(2+(1-\dd)(1-\ad)\bigr)(\ell_{\al,\2}-\ell_{\al,\1})
       +\bigl(2-(1-\dd)(1-\ad)\bigr)(\ell_{\be,\1}+\ell_{\be,\2})\\
=2(1-2\ell_{\al,\1})+(1-\dd)(1-\ad)(2\ell_{\al,\2}-1).
\end{multline*}
With some further work we can compute the constant: if we set $c(\dd,\ad)=1$,
it equals $24/\bigl(4-(1-\dd^2)(1-\ad^2)\bigr)$. We can instead take this
constant as the definition of $c(\dd,\ad)$ (by symmetry, it is the same for
all four components of $\eaf$). For $\eaf_{\al,\1}$, we then have
\begin{align*}
\eaf_{\al,\1}(v_{\1\1})&=2-\ad(1-\dd)(1-\ad)&
\eaf_{\al,\1}(v_{\2\2})&=-2+(1+\dd)(1-\ad)\\
\eaf_{\al,\1}(v_{\1\2})&=2-\dd(1-\dd)(1-\ad)&
\eaf_{\al,\1}(v_{\2\1})&=-2+(1-\dd)(1+\ad)\\
2\eaf_{\al,\1}(v_\dd)&=(1-\ad)(2-(1-\dd)(1+\ad))&
2\eaf_{\al,\1}(v_\ad)&=(1-\dd)(2-(1+\dd)(1-\ad))
\end{align*}
\begin{equation}\label{centroid}
4\eaf_{\al,\1}(v_0)=(1+\dd^2)(1-\ad)+(1+\ad^2)(1-\dd)
\end{equation}
where $v_\dd=\frac12(v_{\1\1}+v_{\2\2})$ and
$v_\ad=\frac12(v_{\1\2}+v_{\2\1})$ are the midpoints of the diagonals, and
$v_0=\frac12(v_\dd+v_\ad)$ is the centroid of $\Delta$.

\begin{lemma}\label{l:unstable} The extremal affine function $\eaf$ is
positive at the centroid of $\Delta$, and is also positive at the midpoints of
the diagonals if $(1+|\dd|)(1+|\ad|)<2$. In general, the value of $\eaf$ at
the midpoint of a diagonal is a positive multiple of the Futaki invariant of a
simple PL convex function with crease along that diagonal.
\end{lemma}
\begin{proof} By~\eqref{centroid} and analogous formulae for the other
components, $\eaf$ is positive at the centroid $v_0$, since $|\dd|,|\ad|<1$.
Similarly, it is positive at the midpoints of the diagonals for
$(1+|\dd|)(1+|\ad|)<2$.  We now compute the Futaki invariant of a simple PL
function $H$ with a crease along the diagonal $v_{\1\1}v_{\2\2}$ through
$v_\dd$.  By symmetry, it suffices to compute the $r_{\al,\1}$ component
$\eaf_{\al,\1}$.  Modulo an affine function (on which the Futaki invariant
vanishes), we may assume $f$ vanishes on $\Fa_{\al,\1}$, so that we only need
to compute the integral of $-f\eaf_{\al,\1}$ over the triangle
$T=v_{\1\1}v_{\2\2}v_{\2\1}$, on which we may suppose
$f=\ell'_{\al,\1}-\ell'_{\be,\1}$, i.e., $f(v_{\2\1})=2/(1+\ad)$.  By
Lemma~\ref{l:simplicial-integral}, the integral evaluates to a universal
constant positive multiple of $-\me(T)f(v_{\2\1})
(\eaf_{\al,\1}(v_{\1\1})+\eaf_{\al,\1}(v_{\2\2})+2\eaf_{\al,\1}(v_{\2\1}))$.
By Lemma~\ref{l:eaf-eqns}, this is a universal constant positive multiple of
\begin{equation*}
-\frac{\bigl(\eaf_{\al,\1}(v_{\1\1})+\eaf_{\al,\1}(v_{\2\2})\bigr)
  \bigl(1-2/(1-\ad)\bigr)}{(1+\dd)(1-\dd)(1+\ad)^2}
=\frac{2\eaf_{\al,\1}(v_\dd)} {(1+\dd)(1-\dd)(1+\ad)(1-\ad)}.
\end{equation*}
Hence the Futaki invariant is a positive multiple of $\eaf(v_\dd)$, as
required. The argument for the other diagonal/midpoint is similar.
\end{proof}

We conclude that if a rational Delzant quadrilateral is K-polystable, then the
extremal affine function must be positive at the diagonal midpoints.

\begin{defn}\label{nice-quads} Let $(\Delta,\Lab)$ be a rational Delzant
quadrilateral with vertices $v_{\1\1}$, $v_{\2\2}$, $v_{\1\2}$, $v_{\2\1}$ so
that $v_{\1\1}$ and $v_{\2\2}$ are opposite.  Then $\Delta$ is said to be
\emph{equipoised} if
$\eaf(v_{\1\1})+\eaf(v_{\2\2})=\eaf(v_{\1\2})+\eaf(v_{\2\1})$, and
\emph{temperate} if $\eaf(v_{\1\1})+\eaf(v_{\2\2})$ and
$\eaf(v_{\1\2})+\eaf(v_{\2\1})$ are both positive.
\end{defn}
The condition to be equipoised was introduced in~\cite{Eveline}, and means
equivalently that the extremal affine function has equal values on the
midpoints $v_\dd$ and $v_\ad$ of the diagonals of $\Delta$. This is automatic
if $\Delta$ is a parallelogram (when $v_\dd=v_\ad$)---otherwise it means that
the extremal affine function is constant (hence positive by
Lemma~\ref{l:unstable}) on the \emph{Newton line} $v_\dd v_\ad$. The condition
to be temperate is the much weaker condition that $\eaf$ is positive on the
line segment between $v_\dd$ and $v_\ad$.

\section{Orbifold compactifications of ambitoric structures}\label{s:orbi-amb}

\subsection{Ambitoric structures and their local classification}

In~\cite{ACG1}, we studied the following $4$-dimensional geometric structure.
\begin{defn} An \emph{ambik\"ahler structure} on a real $4$-manifold or
orbifold $M$ consists of a pair of K\"ahler metrics $(g_-, J_-, \omega_-)$ and
$(g_+, J_+, \omega_+)$ such that
\begin{bulletlist}
\item $g_-$ and $g_+$ induce the same conformal structure (i.e., $g_- =
f^2g_+$ for a positive function $f$ on $M$);
\item $J_-$ and $J_+$ have opposite orientations (equivalently the
volume elements $\frac1{2}\omega_-\wedge\omega_-$ and
$\frac1{2}\omega_+\wedge\omega_+$ on $M$ have opposite signs).
\end{bulletlist}
The structure is said to be \emph{ambitoric} if in addition
\begin{bulletlist}
\item there is a $2$-dimensional subspace $\tor$ of vector fields on $M$,
linearly independent on a dense open set, whose elements are hamiltonian and
Poisson-commuting Killing vector fields with respect to both $(g_-,\omega_-)$
and $(g_+,\omega_+)$.
\end{bulletlist}
\end{defn}
Thus $M$ has a pair of conformally equivalent but oppositely oriented K\"ahler
metrics, invariant under a local $2$-torus action, and both locally toric with
respect that action.
\begin{exs} There are three classes of examples of ambitoric structures.
\begin{numlist}
\item {\bf Toric K\"ahler products}. Let $(\Sigma_1,g_1,J_1,\omega_1)$ and
$(\Sigma_2,g_2,J_2,\omega_2)$ be (locally) toric $2$-dimensional K\"ahler
manifolds or orbifolds, with hamiltonian Killing vector fields $K_1$ and
$K_2$. Then $M=\Sigma_1\times \Sigma_2$ is ambitoric, with $g_\pm=g_1\oplus
g_2$, $J_\pm=J_1\oplus (\pm J_2)$, $\omega_\pm=\omega_1\oplus (\pm\omega_2)$
and $\tor$ spanned by $K_1$ and $K_2$.
\item {\bf Toric Calabi geometries}. Let $(\Sigma,g,J,\omega)$ be a
$2$-dimensional K\"ahler manifold or orbifold with hamiltonian
Killing vector field $K$, let $\pi\colon P\to \Sigma$ be a circle bundle with
connection $\theta$ and curvature $\d\theta=\pi^*\omega$, and let $A$ be a
positive function on an open subset $U$ of $\R^+$. Then $M=P\times U$ is
ambitoric, with
\begin{align*}
g_\pm &= z^{\pm1}
\Bigl(g+\frac{\bigl(z^{-1}\d z\bigr)^2}{A(z)}+ A(z)\theta^2\Bigr),\\
\omega_\pm &= z^{\pm1}\bigl( \omega \pm z^{-1} \d z\wedge\theta\bigr),\qquad
J_\pm (z^{-1} \d z) = \pm A(z) \theta,
\end{align*}
and the local torus action spanned by $K$ and the generator of the circle
action. Here $z\colon M\to \R^+$ is the projection onto $U\sub \R^+$.
\item {\bf Regular ambitoric structures}. Let $q(z)=q_0z^2+2q_1 z+ q_2$ be a
quadratic polynomial and let $M$ be a $4$-dimensional manifold or orbifold
with real-valued functions $(x,y,\tau_0,\tau_1,\tau_2)$ such that $x>y$,
$2q_1\tau_1=q_0\tau_2+q_2\tau_0$, and their exterior derivatives span each
cotangent space. Let $\tor$ be the $2$-dimensional space of vector fields $K$
on $M$ with $\d x(K)=0=\d y(K)$ and $\d\tau_j(K)$ constant, and let $A$ and
$B$ be functions on open neighbourhoods of the images of $x$ and $y$ in $\R$.
Then $M$ is ambitoric with
\begin{align} \label{g-pm}
g_\pm &= \biggl(\frac{x-y}{q(x,y)}\biggr)^{\!\!\pm1}
\biggl(\frac{\d x^2}{A(x)} + \frac{\d y^2}{B(y)}\\ &\notag \qquad\quad
+ A(x) \Bigl(\frac{y^2 \d\tau_0 + 2y \d\tau_1 + \d\tau_2}{(x-y)q(x,y)}\Bigr)^2
+ B(y) \Bigl(\frac{x^2 \d\tau_0 + 2x \d\tau_1 + \d\tau_2}{(x-y)q(x,y)}\Bigr)^2
\biggr),\\
\label{omega-pm}
\omega_\pm &= \biggl(\frac{x-y}{q(x,y)}\biggr)^{\!\!\pm 1}
\frac{\d x\wedge (y^2 \d\tau_0 + 2y \d\tau_1 + \d\tau_2)
\pm \d y \wedge (x^2 \d\tau_0 + 2x \d\tau_1 + \d\tau_2)}{(x-y)q(x,y)},\\
\label{J-pm}
J_\pm &\d x= A(x)\frac{y^2 \d\tau_0 + 2y \d\tau_1 + \d\tau_2}{(x-y)q(x,y)},
\quad
J_\pm \d y= \pm B(y)\frac{x^2 \d\tau_0 + 2x \d\tau_1 + \d\tau_2}{(x-y)q(x,y)},
\end{align}
where $q(x,y)=q_0xy+q_1(x+y)+q_2$.
\end{numlist}
If $(g_\pm,J_\pm,\omega_\pm,\tor)$ is a regular ambitoric structure for which
the quadratic form $q$ has vanishing discriminant, then $(g_+,J_+,\omega_+)$ is
\emph{orthotoric} in the sense of~\cite{ACG}: see~\cite{ACG1}.
\end{exs}
It is easy to check that these explicit geometries are ambitoric
(cf.~\cite{ACG,ACG1}).

\begin{result} \cite{ACG1} Let $(M,g_\pm,J_\pm,\omega_\pm,\tor)$ be an
ambitoric $4$-manifold or orbifold.  Then any point in an open dense subset of
$M$ has a neighbourhood on which $(g_\pm,J_\pm,\omega_\pm,\tor)$ is either a
toric K\"ahler product, a toric Calabi geometry, or a regular ambitoric
structure.
\end{result}

\subsection{Invariant geometry in momentum coordinates}\label{s:momenta}

In order to compactify ambitoric structures using the approach of
section~\ref{s:kahler-compact}, we need to describe the metrics in momentum
coordinates. For toric K\"ahler products, toric Calabi geometries and
orthotoric metrics, this has been done systematically by
E.~Legendre~\cite{Eveline} in her resolution of existence problem for extremal
metrics over \emph{equipoised} rational Delzant quadrilaterals.  Hence we
concentrate on the general regular ambitoric case.

We shall make essential use of the underlying geometry of regular ambitoric
structures. For this we recall (from~\cite{ACG1}) that there is a natural
PSL$(2,\R)$ gauge freedom in regular ambitoric structures, which may be
described in an invariant geometric way by viewing the codomain of the $x$ and
$y$ coordinates as a projective line $\Proj(W)$, where $W$ is a
$2$-dimensional real vector space on which we fix a (nonzero) area form
$\kappa$. We use bold font for the maps $\bx,\by\colon M \to\Proj(W)\sub
\cO(1)\otimes W$ and $x,y\colon M\to \R$ for their expression in an affine
coordinate on $\Proj(W)$; in the affine trivialization of $\cO(1)$, $\bx,\by$
are homogeneous coordinates corresponding to the inhomogeneous coordinates
$x,y$.

The quadratic form $q$ in~\eqref{g-pm}--\eqref{J-pm} is naturally an element
of $S^2W^*$ (i.e., an algebraic section of $\cO(2)\to \Proj(W)$) and the Lie
algebra $\tor$ of the torus is the subspace $S^2_{0,q}W^*$ orthogonal to $q$
with respect to the inner product $\kappa\otimes\kappa$ on $W^*\otimes W^*$
(which restricts to the polarization of the discriminant on $S^2W^*$). Note
that $S^2W^*\cong\mathfrak{sl}(W)$ has a Lie algebra structure (the Poisson
bracket, or Wronskian) and the Poisson bracket with $q$ induces an isomorphism
$\mathrm{ad}_q\colon S^2W^*/\spn{q}\to S^2_{0,q}W^*$. Following~\cite{ACG1},
we distinguish these isomorphs of $\tor$ by writing $\taumap=
\mathrm{ad}_q(\ang)$ where $\taumap$ and $\ang$ take values in
$S^2_{\smash{0,q}}W^*$ and $S^2W^*/\spn{q}$ respectively (modulo corresponding
lattices). For $\bz\in W$, we denote by $\bz^\flat=\kappa(\bz,\cdot)$ the
corresponding element of $W^*$, using ${}^\sharp$ for the inverse isomorphism.

\begin{defn}\label{rough} A regular ambitoric structure is
said to be of \emph{elliptic}, \emph{parabolic} or \emph{hyperbolic} type if
$q$ has (respectively) zero, one or two distinct real roots (on $\Proj(W)$).
\end{defn}

The spaces $\torh_\pm$ of hamiltonian generators of the torus action with
respect to $\omega_\pm$ are readily computed using~\eqref{omega-pm}
(see~\cite[(24)--(25)]{ACG1}); these yield the affine structures
$\iota_\pm(1)\in \torh_\pm$ and the natural momentum maps $\mu^\pm\colon M\to
\torh_\pm{}^{\!*}$ as functions of $\bx$ and $\by$.

\smallbreak\noindent \textbf{Negative structures}. We identify $\torh_-$ with
$S^2_{0,q} W^*\oplus \Wedge^2 W^* \sub W^*\otimes W^*$ and
\begin{equation}\label{mu-minus}
\mu^-(\bx,\by) = -\frac{\bx\otimes\by \mod q^\sharp}{\kappa(\bx,\by)}
\qquad\text{in}\qquad \torh_-{}^{\!*} \cong W\otimes W/\spn{q^\sharp},
\end{equation}
where $q^\sharp\in S^2W$ is dual to $q$ using $\kappa$ (i.e.,
$q^\sharp(\bz^\flat)=q(\bz)$). Modulo a sign convention, the affine structure
is $\kappa$, with $\ip{\mu^-,\kappa}=1$.  For $\gavec\in W$ (or in
$\Proj(W)\sub\cO(1)\otimes W$), we define $\pl\gavec=\gavec^\flat\otimes
q(\gavec,\cdot)$ and $\pr\gavec =q(\gavec,\cdot)\otimes \gavec^\flat$, which
are decomposables in $W^*\otimes W^*$ orthogonal to $q$, and hence in
$\torh_-$.  The contractions of $\pl\gavec$ and $\pr\gavec$ with
$\mu^-(\bx,\by)$ vanish when $\bx=\gavec$ or $\by=\gavec$ respectively. In
affine coordinates on $\Proj(W)$,
\begin{align*}
\pl\ga (x,y)&=(x-\ga)q(\ga,y),& \pr\ga (x,y)&=q(x,\ga)(y-\ga),\\
\ip{\mu^-(x,y),\pl\ga}&=-\frac{(x-\ga)q(y,\ga)}{x-y}, &
\ip{\mu^-(x,y),\pr\ga}&=-\frac{(y-\ga)q(x,\ga)}{x-y}.
\end{align*}
Hence $\pl\ga$ and $\pr\ga$ are dual (i.e., normal) to level surfaces of $x$
and $y$ respectively.

\smallbreak\noindent \textbf{Positive structures}. Here we have $\torh_+\cong
S^2W^*$ and
\begin{equation}\label{mu-plus}
\mu^+(\bx,\by)  = -\frac{\bx\odot \by}{q(\bx,\by)}\qquad\text{in}\qquad
\torh_+{}^{\!*}\cong W\otimes W/\Wedge^2 W\cong S^2W.
\end{equation}
The affine structure is $q$, with $\ip{\mu^+,q}=1$. Decomposables now all have
the form $\w\gavec=\gavec^\flat\otimes\gavec^\flat\in\torh_+ =S^2W^*\sub
W^*\otimes W^*$, for some $\gavec\in W$ (or $\Proj(W)\sub\cO(1)\otimes W$).
The contraction of $\w\gavec$ with $\mu^+(\bx,\by)$ vanishes when $\bx=\gavec$
or $\by=\gavec$. In affine coordinates on $\Proj(W)$, $\w\ga$ polarizes the
quadratic $\w\ga(z)=(z-\ga)^2$, with
\begin{equation*}
\w\ga(x,y)=(x-\ga)(y-\ga),\qquad
\ip{\mu^+(x,y),\w\ga}= -\frac{(x-\ga)(y-\ga)}{q(x,y)}.
\end{equation*}
Hence $\w\ga$ are dual (i.e., normal) to level surfaces of $x$ and $y$.

\medbreak
The constants in $\torh_\pm$ are elements of $\spn{\kappa}=\Wedge^2W^*$ and
$\spn{q}\sub S^2W^*$ respectively, and the map $w\mapsto
p=\mathrm{ad}_q(w)\colon S^2W^*\to S^2_{0,q}W^*$ sends a Killing potential
with respect to $\omega_+$ to a Killing potential for the same vector field
with respect to $\omega_-$. We denote by $K^{(p)}$ the corresponding vector
field on $M$.

It is now straightforward to compute the torus metrics ${\mathbf H}^\pm$ of
$g_\pm$:
\begin{align}
{\mathbf H}^-_{\mu^-(x,y)}(p,\tilde p)&=g_- (K^{(p)}, K^{(\tilde p)})
=\frac{A(x) p(y) \tilde p(y) + B(y) p(x) \tilde p(x)}{(x-y)^3\, q(x, y) },\\
{\mathbf H}^+_{\mu^+(x,y)}(p,\tilde p)&=g_+ (K^{(p)}, K^{(\tilde p)})
=\frac{A(x) p(y) \tilde p(y) + B(y) p(x) \tilde p(x)}{(x-y)\, q(x, y)^3 },
\end{align}
where $p, \tilde p$ in $S^2_{0,q} W^*\cong \tor$.  Up to a constant multiple
(depending on a choice of basis for $\tor^*$), we have
\begin{equation}\label{detH}
\det {\mathbf H}^-_{\mu^-(x,y)} = \frac{A(x) B(y)}{(x-y)^4},\qquad\qquad
\det {\mathbf H}^+_{\mu^+(x,y)} = \frac{A(x) B(y)}{q(x,y)^4}.
\end{equation}

\subsection{Orbifold compactifications of ambitoric K\"ahler surfaces}
\label{s:ambi-orbifolds}

The existence of an ambitoric structure on a compact $4$-orbifold $M$ places
strong (and well-known) constraints on the topology of $M$.

\begin{prop}\label{pr:quad} Let $M$ be a compact connected $4$-orbifold with
an effective action of a $2$-torus $\T$, and suppose that
$(g_\pm,J_\pm,\omega_\pm)$ is an ambitoric structure on $M$ with respect to
the derivative $\tor\hookrightarrow C^\infty(M,TM)$ of the $\T$ action.  Then
the images of the momentum maps of the $\T$ action \textup(with respect to
$\omega_+$ and $\omega_-$\textup) are quadrilaterals \textup(i.e.,
$b_2(M)=2$\textup). In particular, if $M$ is smooth, then for some $k\in\N$,
$(M,J_+)$ and $(M,J_-)$ are biholomorphic to a Hirzebruch surface $\Proj(\cO
\oplus \cO(k))\to \C P^1$.
\end{prop}
\begin{proof} If $M$ is a compact K\"ahler surface admitting a holomorphic
hamiltonian action of a $2$-torus $\T$ and $M^0$ is the union of the generic
$\T$-orbits, then the anticanonical bundle has a holomorphic section with
zeroset $M\setminus M^0$. The canonical bundle therefore has no holomorphic
sections, and so $h^{2,0}=h^{0,2}=0$. In the ambitoric case, the only
nonvanishing second deRham cohomology is in the intersection of $H^{1,1}$ with
respect to $J_+$ and $J_-$, hence represented by a constant linear combination
of the harmonic forms $\omega_+$ and $\omega_-$ (since $g_+$ and $g_-$ are
conformally equivalent). It follows that $M$ has second Betti number $b_2(M)=2$
($b_+(M)=b_-(M)=1$).  Standard results about compact simply connected
$4$-orbifolds with $2$-torus actions (e.g.~\cite{BGMR}) then imply that the
rational Delzant polygons have $b_2(M)+2=4$ sides.
\end{proof}
\begin{rem} Conversely, it is well-known that any Hirzebruch surface $M$
admits compatible ambitoric structures: indeed it admits toric \emph{extremal}
metrics of Calabi type in each K\"ahler class~\cite{Cal1}, where the base
metric $g_\Sigma$ is the Fubini--Study metric on $\C P^1$---such metrics are
ambitoric by \cite[Proposition~9]{ACG1}.  Furthermore, any ambik\"ahler
structure $(g_\pm,\omega_\pm)$ on $M$ with $g_+$ or $g_-$ extremal is of this
type: if $g_+$ is extremal, then by uniqueness for extremal K\"ahler metrics
in their K\"ahler class~\cite{CT}, we may assume $g_+$ is of Calabi type,
hence ambitoric with respect to a negative complex structure ${\tilde J}_-$.
However, $g_+$ cannot have selfdual Weyl tensor, so $J_-$ must equal $\pm
{\tilde J}_-$.
\end{rem}

Note that we have assumed above that both K\"ahler metrics $(g_\pm,
J_\pm,\omega_\pm)$ are globally defined on $M$, not just at points in generic
$\T$-orbits. In the following we shall weaken this assumption slightly.

\begin{defn}\label{compactification} An \emph{ambitoric compactification}
is a compact connected oriented $4$-orbifold $M$ with an effective action of a
$2$-torus $\T$ such that on the (dense) union $M^0$ of the free $\T$-orbits,
there is an ambitoric structure $(g_\pm, J_\pm,\omega_\pm,\tor)$ (with $\tor$
the Lie algebra of $\T$) for which at least one of the K\"ahler metrics
extends smoothly to a toric K\"ahler metric on $(M,\T)$. An ambitoric
compactification is \emph{regular} if the ambitoric structure on $M^0$ is
regular with $(x,y)$-coordinates that are globally defined on $M^0$.
\end{defn}

Henceforth, we consider only regular ambitoric compactifications (without loss
of generality if we are interested in extremal ambitoric metrics, as
Theorem~\ref{thm:stability} below will show). We say the ambitoric
compactification $M$ is positive and/or negative if $g_+$ and/or $g_-$ extends
smoothly to $M$. ($M$ can be both positive and negative.)

By Proposition~\ref{p:toric-ccs}, if $g_\pm$ compactifies, the
determinant~\eqref{detH} of $\mathbf{H}^\pm$ must be smooth on $M$, positive
on $M^0$, and vanish on (the pre-image of) the boundaries of
$\Delta_\pm$. Hence the image of $M^0$ under $(x,y)$ must be a domain
$D^0:=(\al_\1,\al_\2)\times (\be_\1,\be_\2)$ where $A(z)$ and $B(z)$ are
positive on $(\al_\1,\al_\2)$ and $(\be_\1,\be_\2)$ respectively, with zeros
at the endpoints; furthermore, if $g_+$ and/or $g_-$ are globally defined,
then $q(x,y)\neq 0$ and/or $x-y\neq 0$ on the closure $D=[\al_\1,\al_\2]\times
[\be_\1,\be_\2]$ of $D^0$. If both $g_+$ and $g_-$ are globally defined,
$\al_\1>\be_\2$ and $\Delta_\pm$ are both quadrilaterals.  However, in order
to apply limiting arguments, we also need to allow $\be_\2=\al_\1$ when
$\Delta_+$ is a simplex and $M$ is a weighted projective plane. In this case
$g_-$ does not compactify.

The polytopes $\Delta_\pm\sub \torh_\pm{}^{\!*}$ are the images of
$D\sub\Proj(W)\times\Proj(W)$ (using the chosen affine chart on $\Proj(W)$)
under the formulae~\eqref{mu-minus}--\eqref{mu-plus} for the momenta
$\mu^\pm(\bx,\by)$. Since the level surfaces $x=\ga$, $y=\ga$ have normals
$\pl\ga$ and $\pr\ga$ (respectively) in the negative case, and $\w\ga$ in
the positive case, we can take $\ga=\al_\1,\al_\2,\be_\1$ and $\be_\2$ to
determine straightforwardly that
\begin{align*}
\Delta_-=\Bigl\{\xi\in \torh_-{}^{\!*}:\ip{\xi,\kappa}=1,\,
\ip{\xi,\pl{\al_\1}}\leq 0,\, &\ip{\xi,\pl{\al_\2}}\geq 0,\\
\ip{\xi,\pr{\be_\1}}\leq 0,\, &\ip{\xi,\pr{\be_\2}}\geq 0\Bigr\};\\
\Delta_+=\Bigl\{\xi\in \torh_+{}^{\!*}:\ip{\xi,q}=1,\,
\ip{\xi,\w{\al_\1}}\leq0,\,
&\ip{\xi,\w{\al_\2}}\geq0,\\
\ip{\xi,\w{\be_\1}}\geq0,\,&\ip{\xi,\w{\be_\2}}\leq0\Bigr\}.
\end{align*}

Normals to $\Delta_\pm$ may be written $u^\pm_{\al,k}=\p{\al_k}/r^\pm_{\al,k}$
and $u^\pm_{\be,k}=\p{\be_k}/r^\pm_{\be,k}$ for $k=\1,\2$ and constants
$r^\pm_{\al,k}$ and $r^\pm_{\be,k}$, where $\p\gavec=\gavec^\flat\odot
q(\gavec,\cdot)\in \tor$ has the affine expression
\begin{equation*}
\p\ga(x,y) = \tfrac12 \bigl(q(x,\ga)(y-\ga)+(x-\ga)q(y,\ga)\bigr).
\end{equation*}
The boundary conditions ${\mathbf H}^\pm_{\mu^\pm(\al_k,y)}
(u^\pm_{\al,k},\cdot)=0={\mathbf H}^\pm_{\mu^\pm(x,\be_k)}
(u^\pm_{\be,k},\cdot)$ (see \eqref{toricboundary}) are equivalent to
$A(\al_k)=0=B(\be_k)$, and the remaining boundary conditions simplify to
$A'(\al_k)=2r^\pm_{\al,k}$ and $B'(\be_k)=\mp 2r^\pm_{\be,k}$, using e.g.,
\begin{equation*}
\d{\mathbf H}^+_{\mu^+(\al_\1,y)}(\p{\al_\1},\p{\al_\1})
=\frac{A'(\al_\1) (\p{\al_\1}(y))^2}{(\al_\1-y) q(\al_\1,y)^3} \d x
=\frac{A'(\al_\1) (\al_\1-y)}{q(\al_\1, y)} \d x,
\end{equation*}
which is equal to $-\frac 12 A'(\al_\1) \p{\al_\1}$. We deduce that
$r_{\al,k}:=r^+_{\al,k}=r^-_{\al,k}$ and $r_{\be,k}:=r^+_{\be,k}=
-r^-_{\be,k}$. The construction of (simply connected) regular ambitoric
compactifications is now completed by ensuring the normals are inward, and
satisfy the integrality condition that they span a lattice.

\begin{prop}\label{p:compactification}  Any compact, simply connected regular
ambitoric compactification is determined by the following data\textup:
\begin{bulletlist}
\item real numbers $\al_k, \be_k, r_{\al,k}, r_{\be,k}$ $(k=\1,\2)$, subject
to the inequalities
\begin{equation*}
\be_\1 < \be_\2 \leq \al_\1 < \al_\2, \qquad
r_{\al,\1} < 0 < r_{\al,\2}, \quad r_{\be,\1} > 0 > r_{\be,\2},
\end{equation*}
and the integrality condition that, with $u_{\al,k}=\p{\al_k}/r_{\al,k}$ and
$u_{\be,k}=\p{\be_k}/r_{\be,k}$,
\begin{equation}\label{rational}
\spns_{\Z} \{u_{\al,\1}, u_{\al,\2}, u_{\be,\1}, u_{\be,\2}\}\cong \Z^2.
\end{equation} 
\item a quadratic $q(z)$ and two smooth functions of one variable, $A(z)$ and
$B(z)$, satisfying the positivity conditions that $q(x,y)>0$ on
$D^0=(\al_\1,\al_\2)\times (\be_\1,\be_\2)$, $A(z)>0$ on
$(\al_\1,\al_\2)$ and $B(z)>0$ on $(\be_\1, \be_\2)$, and the boundary
conditions that
\begin{equation}\label{AB-boundary}
A(\al_k)=0=B(\be_k),\qquad A'(\al_k)=-2r_{\al,k}, \quad B'(\be_k)
=2r_{\be,k} \qquad\quad (k=\1,\2).
\end{equation}
\end{bulletlist}
It is positive if $q(x,y)>0$ on the closure $D$ of $D^0$, and negative if
$\be_\2<\al_\1$.
\end{prop}
\begin{rem}\label{rational-polytope} A particular case where \eqref{rational}
holds automatically is when $q$ has rational coefficients and $\al_k, \be_k$
and $r_{\al,k}, r_{\be,k}$ are all rational: since the condition
\eqref{rational} is clearly invariant under an overall multiplication of
$r_{\al,k}$ and $r_{\be,k}$ by a nonzero real constant, we can choose this
constant such that $u_{\al,k}$ and $u_{\be,k}$ have integer coordinates.
\end{rem}

\begin{rem}\label{zero-measure} One can allow some (but not all) of $r_{\al,k}$
and $r_{\be,k}$ in Proposition~\ref{p:compactification} to be zero. In terms
of the theory reviewed in section~\ref{s:toric-orb}, this is a limiting case
in which some of the normals $u_j$ are infinite, and hence the measure $\d\nu$
on the corresponding facet $\Fa_j$ is zero. On such an ``omitted'' facet
$\Fa_j$, the first order boundary conditions of Proposition~\ref{p:toric-ccs}
become
\begin{equation}\label{complete-boundary}
\mathbf{H}_{\xi}({\tilde u}_j, \cdot)=0 \quad {\rm and} \quad
(\d\mathbf{H})_{\xi}({\tilde u}_j, {\tilde u}_j)=0, \quad \forall \xi\in\Fa_j,
\end{equation}
where ${\tilde u}_j$ is \emph{any} nonzero normal vector to $\Fa_j$. This is
the setting of \cite[Conjecture 7.2.3]{Donaldson1} and \cite[\S 3.1]{Sz2} and
yields complete K\"ahler metrics on the complement of a toric divisor (the
inverse image of the omitted facets) in a compact toric orbifold $M$.

Proposition~\ref{p:compactification} extends the characterization of regular
ambitoric compactifications by allowing complete ends of this form. When
$r_{\al,k}$ or $r_{\be,k}$ is zero, the boundary
conditions~\eqref{AB-boundary} apply without change, but in~\eqref{rational},
we replace $u_{\al,k}=\p{\al_k}/r_{\al,k}$ or $u_{\be,k}=\p{\be_k}/r_{\be,k}$
by some other multiples of $\p{\al_k}$ or $\p{\be_k}$.
\end{rem}

\subsection{Factorization structures for triangles and quadrilaterals}

In the previous two subsections, we found that in an ambitoric
compactification, the coordinate lines (in particular, the facets) in
$\Delta\sub\Proj(\torh^*)$ were dual to (projective) normals in $\Proj(\torh)$
which were decomposable with respect to a inclusion of $\torh$ into a tensor
product of $2$-dimensional vector spaces $W^*\otimes W^*$. In order to obtain
a converse, and determine when a rational Delzant quadrilateral arises from an
ambitoric compactification, we formalize this phenomenon by introducing
($2$-dimensional) factorization structures; for a more general context see
Appendix~\ref{s:fs}.

Throughout this section, we adopt the notation of \S\ref{s:quad} for
$(\Delta,\Lab)$ in $\Proj(\torh^*)$: the affine normals will be indexed
$L_{\be,\1}$, $L_{\al,\1}$, $L_{\be,\2}$, $L_{\al,\2}$. We also allow the
quadrilateral to degenerate to a triangle with $L_{\al,\1}=L_{\be,\2}$.

\begin{defn}\label{factorization} Let $(\Delta,\Lab)$ be a rational Delzant
quadrilateral or a triangle in $\Proj(\torh^*)$, let $W_1,W_2$ be
$2$-dimensional vector spaces, and let $S\colon \Proj(W_1)\times\Proj(W_2)\to
\Proj(W_1\otimes W_2)$ be the Segre embedding, sending $([w_1],[w_2])$ to
$[w_1\otimes w_2]$. A \emph{factorization structure} is a rational map
$S_\fs\colon \Proj(W_1)\times\Proj(W_2)\dashrightarrow \Proj(\torh^*)$
obtained by composing $S$ with a projection $\Proj(W_1\otimes W_2)
\dashrightarrow \Proj(\torh^*)$ dual to a linear injection $\fs\colon \torh\to
W_1^*\otimes W_2^*$.
\begin{numlist}
\item $S_\fs$ is a \emph{Segre factorization structure} if the image of $\fs$
is ${\gavec_1}^0\otimes W_2^*+W_1^*\otimes {\gavec_2}^0 \subset W_1^*\otimes
W_2^*$, where ${\gavec_j}\sub W_j^*$ is the annihilator of some $\gavec_j\in
W_j$(for $j=1,2$).

\begin{center}
\includegraphics[width=0.8\textwidth]{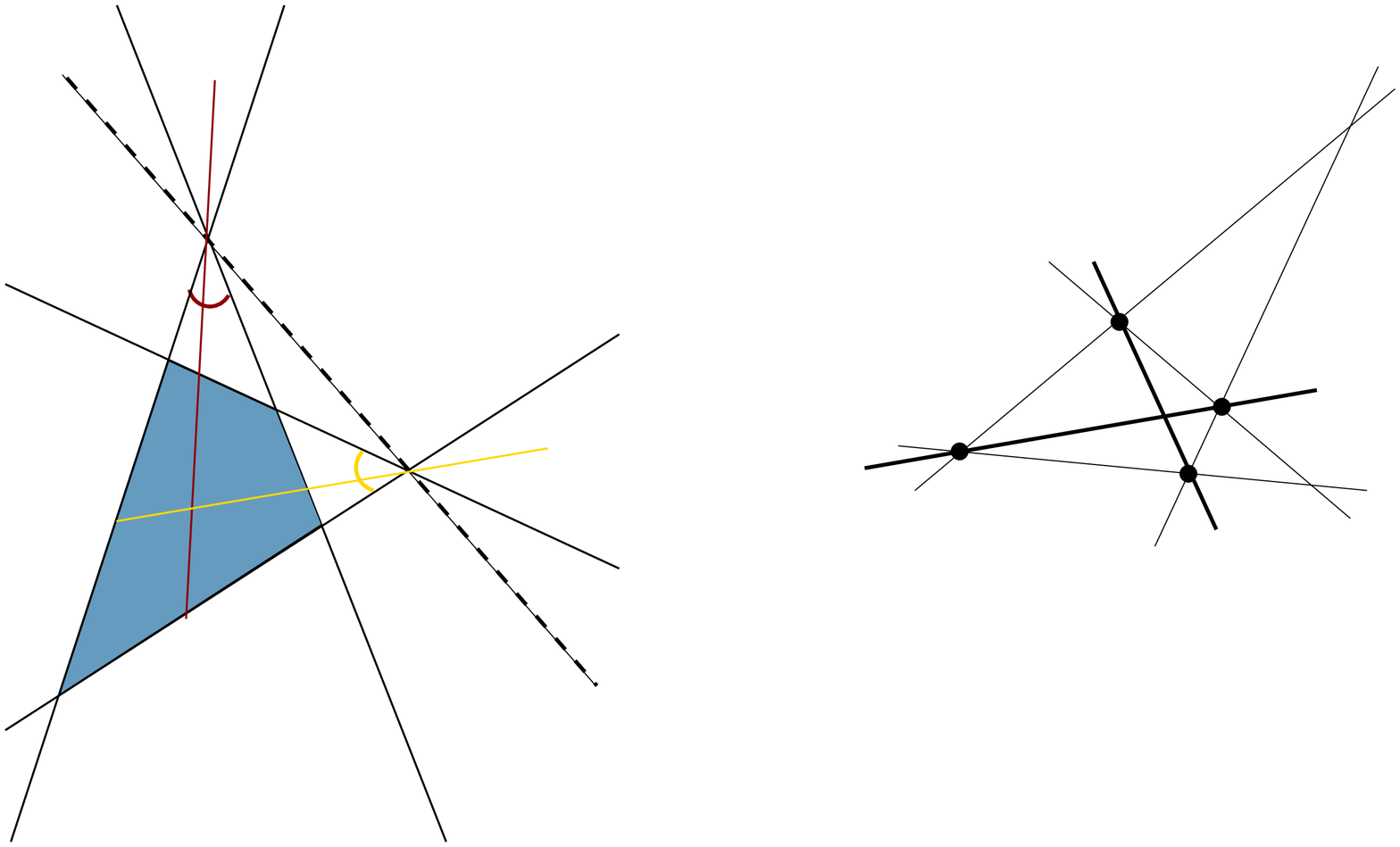}

Figure 2. Segre factorization in $\Proj(\torh^*)$ and dual conic in
$\Proj(\torh)$.
\end{center}

\item $S_\fs$ is a \emph{Veronese factorization structure} if there is an
isomorphism $W_1\cong W_2$ (so we drop the index) such that the image of $\fs$
is $S^2W^*\subset W^*\otimes W^*$.

\begin{center}
\includegraphics[width=0.8\textwidth]{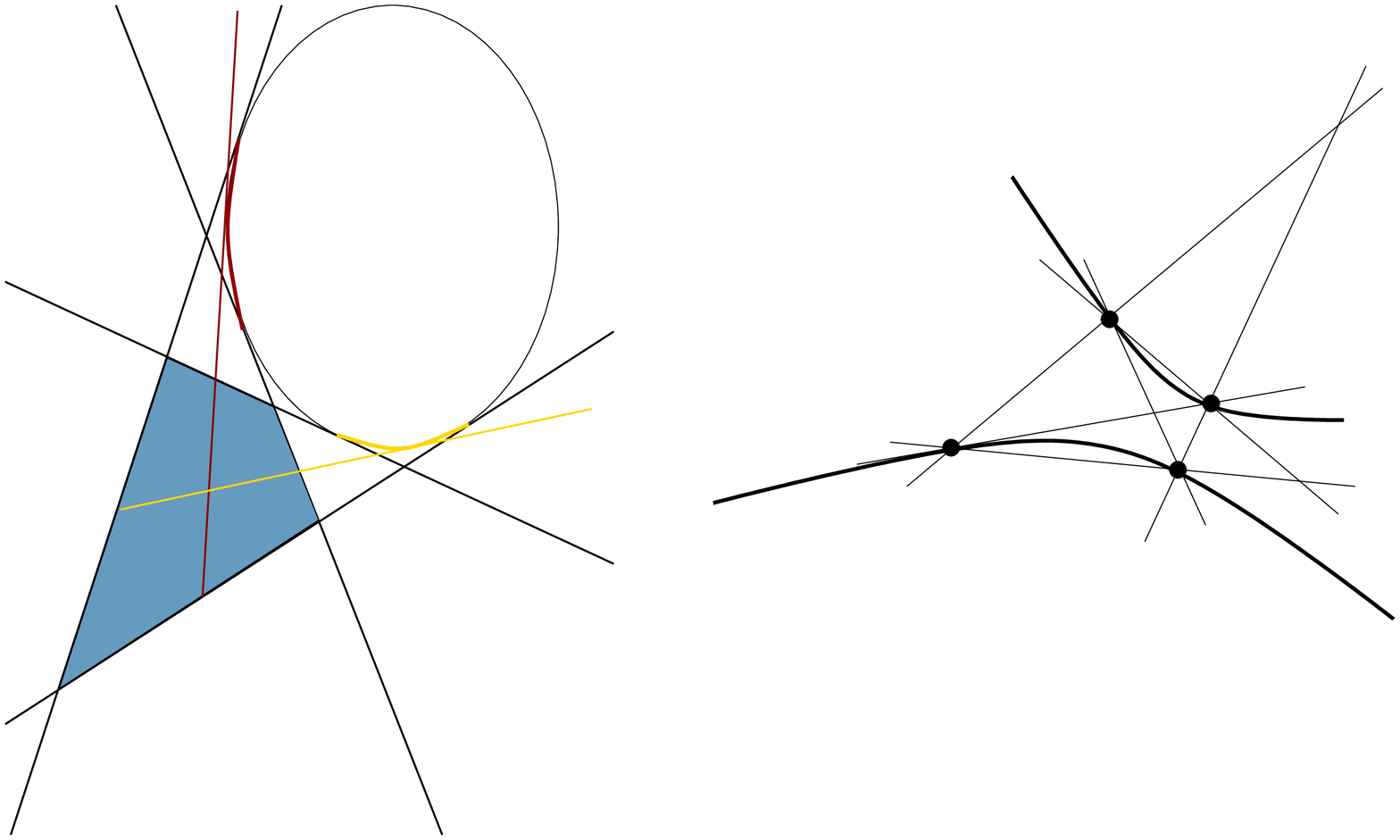}

Figure 3. Veronese factorization in $\Proj(\torh^*)$ and dual conic in
$\Proj(\torh)$.
\end{center}
\end{numlist}
We say $S_\fs$ is \emph{compatible} with $(\Delta,\Lab)$ if $S_\fs$ maps a
product $I_1\times I_2$ of closed intervals in $\Proj(W_1)\times \Proj(W_2)$
bijectively onto $\Delta$.
\end{defn}

The image of the Segre embedding $S$ is a nonsingular ruled quadric surface,
and the pullback by $\fs$ is a hyperplane section. This is a conic $\cC$ in
$\Proj(\torh)$ which we call the \emph{induced conic} of the factorization
structure; it is a line-pair in the Segre case, and nonsingular (and nonempty)
in the Veronese case. Conversely, any two nonsingular ruled quadric surfaces
are projectively equivalent, as are any two line-pairs, or any two nonsingular
nonempty conics, in a projective plane. Hence a Segre or Veronese
factorization structure is determined up to isomorphism by a line-pair or
nonsingular nonempty conic $\cC$ in $\Proj(\torh)$.

Such a conic $\cC$ in $\Proj(\torh)$ has a dual $\cC^*$ in $\Proj(\torh^*)$:
in the Segre case (Figure 2), $\cC^*$ is a ``double'' line (dual to the
vertex of the line-pair $\cC$) with two marked points (dual to the two lines),
while in the Veronese case (Figure 3), $\cC^*$ is the conic of tangent
lines to $\cC$. The coordinate lines of the factorization structure (whose
duals, i.e., projective normals, are the points of $\cC$) are
\begin{itemize}
\item the lines through the two marked points on $\cC^*$ in the Segre case;
\item the tangent lines to $\cC^*$ in the Veronese case.
\end{itemize}

For compatibility with $(\Delta,\Lab)$, the projective normals $[L_{\be,\1}]$,
$[L_{\al,\1}]$, $[L_{\be,\2}]$, and $[L_{\al,\2}]$ must be on $\cC$ (so that
the facets of $\Delta$ are on coordinate lines). This is not quite sufficient:
in the Segre case, $\cC^*$ must not meet the interior of $\Delta$, while in
the Veronese case, $\Delta$ must be entirely in the ``exterior'' (union of
tangent lines) of $\cC^*$. Using the projectivized dual cone $\Delta^*\sub
\Proj(\torh)$, the following ensures both requirements.

\begin{cond}\label{conic-meets} $\cC$ meets the interior of $\Delta^*$.
\end{cond}
When $\Delta$ is a quadrilateral (bounded by four lines in general position),
there is a pencil of conics through the four projective normals, and $\cC$ can
be any conic in the pencil satisfying Condition~\ref{conic-meets}.

To complete our analysis, we need to discuss what happens to the affine
structure $\iota(1)\in \torh$ under the factorization structure. In the Segre
case, there are three possibilities for $\fs(\iota(1))$: if $\fs(\iota(1))\in
{\gavec_1}^0\otimes{\gavec_2}^0$, then $\cC^*$ is the line at infinity and
$\Delta$ is a parallelogram; otherwise $\fs(\iota(1)$ is either decomposable,
in which case $\Delta$ is a trapezium (two parallel sides), or indecomposable,
in which case $\Delta$ has no parallel sides. In the Veronese case,
$\fs(\iota(1))=q\in S^2W^*$ and there are also three possibilities: $q$
may have positive, zero or negative discriminant.

\begin{prop}\label{p:existence} Let $(\Delta,\Lab)$ be a rational Delzant
quadrilateral, and $\cC$ a conic through the projective normals of $\Delta$
which satisfies Condition~\textup{\ref{conic-meets}}.  Then if $\cC$ is
nonsingular \textup(respectively $[\iota(1)]$ is not on $\cC$\textup), there
is a positive \textup(respectively negative\textup) ambitoric compactification
with rational Delzant polytope $\Delta$ and induced conic $\cC$.
\end{prop}
\begin{proof} For $\cC$ nonsingular, Condition~\ref{conic-meets} implies there
is a Veronese factorization $\fs$ compatible with $(\Delta,\Lab)$. We identify
$\torh$ with $S^2W^*$ using $\fs$, and fix an area form $\kappa$ on $W$, hence
an isomorphism $\gavec\mapsto \gavec^\flat\in \gavec^0$ from $W$ to $W^*$. Up
to an overall sign, the affine normals therefore have the form
$L_{\be,\1}=-{\bevec_\1}^\flat\otimes{\bevec_\1}^\flat$,
$L_{\al,\1}={\alvec_\1}^\flat\otimes{\alvec_\1}^\flat$,
$L_{\be,\2}={\bevec_\2}^\flat\otimes{\bevec_\2}^\flat$ and
$L_{\al,\2}={\alvec_\2}^\flat\otimes{\alvec_\2}^\flat$; it is straightforward
to check that $[\bevec_\1]$, $[\bevec_\2]$, $[\alvec_\1]$, $[\alvec_\2]$ are
in cyclic order on $\Proj(W)$ and hence choose an affine chart in which they
are represented by $\be_\1<\be_\2\leq\al_\1<\al_\2$.
Proposition~\ref{p:compactification} now implies that there is a positive
ambitoric compactification with rational Delzant polytope isomorphic to
$\Delta$ and induced conic $\cC$.

For negative compactifications, we identify $\cC\sub\Proj(\torh)$ with the
conic of decomposables in $\psi^0\sub W_1^*\otimes W_2^*$, where $\psi\in
W_1\otimes W_2$ is decomposable if $\cC$ is singular, and indecomposable
otherwise. (By Appendix~\ref{s:fs4}, the factorization is Veronese if $\cC$ is
nonsingular and Segre otherwise; in the singular case, $\cC$ must be a
line-pair, since the projective normals are not collinear.)  Now if
$\iota(1)\in \torh$ is not on $\cC$, its image in $W_1^*\otimes W_2^*$ is not
decomposable, and may be used to identify $W_1$ and ${W_2}^*$; we may identify
$W_1$ with $W_2$ and drop subscripts by fixing also an area form
$\kappa$. Then $\psi$ is dual to a quadratic form $q\in S^2W^*$, i.e.,
$\torh=q^\perp$. We conclude, similarly to the positive case, that
Condition~\ref{conic-meets} implies that the normals have the form
$L_{\be,\1}=q({\bevec_\1},\cdot)\otimes {\bevec_\1}^\flat$,
$L_{\al,\1}={\alvec_\1}^\flat\otimes q({\alvec_\1},\cdot)$,
$L_{\be,\2}=-q(\bevec_\2,\cdot)\otimes {\bevec_\2}^\flat$ and
$L_{\al,\2}=-{\alvec_\2}^\flat\otimes q({\alvec_\2},\cdot)$; the rest of the
construction, using Proposition~\ref{p:compactification}, follows the positive
case.
\end{proof}
\begin{rem} Both types of compactification exist if $\cC$ is nonsingular and
does not pass through $[\iota(1)]$. They are then related by interchanging the
roles of $q$ (nonsingular) and $\kappa$. If neither exists, $\cC$ is the
line-pair joining opposite normals, and $[\iota(1)]$ lies on one of its
lines. In particular, $\Delta$ is a trapezium or parallelogram.
\end{rem}

Condition~\ref{conic-meets} means that the conic $\cC^*$ tangent to the four
lines of $\Delta$ (the dual conic if $\cC$ is nonsingular, or the double line
dual to the vertex of $\cC$ if it is a line-pair) does not meet $\Delta$,
i.e., $\cC^*$ is not an inscribed ellipse, or a degeneration of such an
ellipse to a double line through opposite points of $\Delta$. We conclude this
section by using the affine structure $\iota(1)\in\Delta^*$ to provide a
sufficient criterion for Condition~\ref{conic-meets}.

\begin{prop}\label{fact-condition} Let $(\Delta,\Lab)$ be a rational Delzant
quadrilateral with affine structure $\iota(1)\in \torh$ and let $\cC$ be a
conic in the pencil through the four projective normals such that $[\iota(1)]$
is not a singular point on $\cC$. Suppose there is an affine function
orthogonal to $\iota(1)$ with no zero on the segment of the Newton line
between the midpoints of the diagonals of $\Delta$.  Then $\cC$ satisfies
Condition~\textup{\ref{conic-meets}}.
\end{prop}
\begin{proof} The centre of $\cC^*$ is the point dual to the line orthogonal
to $[\iota(1)]$ with respect to $\cC$ (or the midpoint of the vertices of
$\Delta$ on $\cC^*$ if $\cC$ is singular). Thus all affine functions
orthogonal to $\iota(1)$ with respect to $\cC$ vanish there. Newton's theorem
for convex quadrilaterals implies that inscribed ellipses (or their
degenerations) have centres (or midpoints) on the segment of the Newton line
between the midpoints of the diagonals of $\Delta$. Hence $\cC^*$ cannot be
among these, so Condition~\ref{conic-meets} holds.
\end{proof}

\section{Extremal ambitoric $4$-orbifolds and convex quadrilaterals}
\label{s:ext-orbi-amb}

\subsection{Extremal ambitoric metrics and adapted factorizations}
\label{s:extremal}

A toric K\"ahler metric is extremal if and only if its scalar curvature is
equal to the extremal affine function. For regular ambitoric structures, a
straightforward computation of the scalar curvatures of the two K\"ahler
metrics yields the following result~\cite{ACG1}.

\begin{result} Let $(J_+,J_-, g_+,g_-,\tor)$ be a regular ambitoric structure
given by a quadratic $q$ and functions of one variable $A,B$. Then $(g_+,J_+)$
is an extremal K\"ahler metric if and only if $(g_-,J_-)$ is an extremal
K\"ahler metric if and only if
\begin{equation}\begin{split}
A(z)&=q(z)\pi(z)+P(z),\\
B(z)&=q(z)\pi(z)-P(z),\\
\end{split}\end{equation}
where $\pi$ is a polynomial of degree at most two orthogonal to $q$ and $P$ is
polynomial of degree at most four. In this case
\begin{align}
s_- &= -\frac{24 \pi(x,y)}{x-y},\\
s_+ &=-\frac{w(x,y)}{q(x,y)},
\end{align}
where the quadratic $w$ \textup(defining $w(x,y)=w_0xy+w_1(x+y)+w_2$\textup)
is equal to $C_q(P)$, where $C_q$ is a surjective linear map from quartics to
quadratics orthogonal to $q$.
\end{result}
In~\cite{ACG1}, we also proved that the ambitoric structure is locally
conformally Einstein if in addition the quadratics $\pi$ and $w$ in this
theorem are linearly dependent. We shall use this to construct examples later.
We shall also need the explicit formula for $C_q(P)(z)$, which is the Poisson
bracket of $q(z)$ with $q(z)P''(z)-3q'(z)P(z)+6q''(z)P(z)$.

For ambitoric compactifications, we deduce the following from the above
theorem.

\begin{cor} For an extremal regular ambitoric compactification with induced
conic $\cC$, the extremal affine function $\eaf_\pm\in\torh_\pm$ of
$\omega_\pm$ is orthogonal to the affine structure $\iota(1)\in\torh_\pm$
with respect to $\cC$.
\end{cor}

Given a rational Delzant quadrilateral $(\Delta,\Lab)$ in $\Xi\subset
\Proj(\torh^*)$, there is a unique conic $\cC(\Delta,\Lab)\subset
\Proj(\torh)$ in the pencil through the normals such that $[\iota(1)]$ is
orthogonal to $[\eaf]$. Now $\cC(\Delta,\Lab)$ corresponds to a conic in
$\Proj(\torh^*)$ such that $\eaf$ vanishes at its centre (if nonsingular) or
midpoint (if a double line).

\begin{lemma}\label{l:cond1} A rational Delzant quadrilateral $(\Delta,\Lab)$
is equipoised iff $[\iota(1)]$ lies on the conic $\cC(\Delta,\Lab)$ and
temperate iff the conic $\cC(\Delta,\Lab)$ satisfies
Condition~\textup{\ref{conic-meets}}.
\end{lemma}

Consequently, for temperate rational Delzant quadrilaterals, there is an
ambitoric compactification unless $\cC(\Delta,\Lab)$ is the diagonal line
pair, and the affine structure $[\iota(1)]$ lies on one of the diagonals,
i.e., $\Delta$ is an equipoised trapezium.

\subsection{Extremal ambitoric orbifolds and K-polystability}
\label{s:stability}

In order to construct extremal ambitoric orbifolds, we specialize the
discussion of section~\ref{s:ambi-orbifolds} to the case that $A(z)$ and
$B(z)$ in Proposition~\ref{p:compactification} are polynomials of degree $\le
4$. Our approach follows~\cite{ACGT,ACGT2,Eveline,Lejmi} to which we
refer the reader for further details.

The boundary conditions~\eqref{AB-boundary} have the general solution
\begin{align*}
A(z)&= (z-\al_\1)(z-\al_\2)\bigl( (c + d)(z-\al_\1)(z-\al_\2)
+ N_{\al,\1}(z-\al_\2) + N_{\al,\2}(z-\al_\1) \bigr)\\
B(z)&= (z-\be_\1)(z-\be_\2)\bigl( (c - d) (z-\be_\1)(z-\be_\2)
+ N_{\be,\1} (z-\be_\2) + N_{\be,\2} (z-\be_\1) \bigr)
\end{align*}
for $A(z)$ and $B(z)$ in terms of $(\al_k,r_{\al,k})$ and $(\be_k,r_{\be,k})$
(for $k=\1,\2$), where $N_{\al,k}=2r_{\al,k}/(\al_\2-\al_\1)^2$ and
$N_{\be,k}=2r_{\be,k}/(\be_\2-\be_\1)^2$ (for $k=\1,\2$), and $c,d$ are two
free parameters. For fixed $q(z)$, the extremality conditions of
section~\ref{s:extremal} state that $A(z)+B(z)=q(z)\pi(z)$ with $\pi$
orthogonal to $q$. These impose three further linear conditions on $A$ and
$B$, which we may solve for $c$ and $d$, leaving one linear condition on
$(r_{\al,k},r_{\be,k})$ whose coefficients depend rationally on $\al_k$ and
$\be_k$ ($k=\1,\2$).

\begin{ex}\label{orthotoric} When $q(z)=1$ (the orthotoric case), we have
$c=0$ and two formulae for $d$ whose equality yields the equation
\begin{equation*}
(N_{\al,\1}\ N_{\al,\2}\ N_{\be,\1}\ N_{\be,\2})
\begin{pmatrix}
(\al_\1+\al_\2-\be_\1-\be_\2)^2 + 2(\al_\2-\be_\1)(\al_\2-\be_\2)\\
(\al_\1+\al_\2-\be_\1-\be_\2)^2 + 2(\al_\1-\be_\1)(\al_\1-\be_\2)\\
(\al_\1+\al_\2-\be_\1-\be_\2)^2 + 2(\be_\2-\al_\1)(\be_\2-\al_\2)\\
(\al_\1+\al_\2-\be_\1-\be_\2)^2 + 2(\be_\1-\al_\1)(\be_\1-\al_\2)
\end{pmatrix}=0
\end{equation*}
found by E. Legendre~\cite{Eveline}.\footnote{Hence the linear system always
  determines $c$ and $d$, since this condition is open and natural in $q$.}
She proved that this condition on $(r_{\al,k},r_{\be,k})$ is equivalent to
$\Delta_+$ being equipoised (relative to the corresponding normals) and thus
showed that the existence of extremal K\"ahler metrics is equivalent to
(toric) K-polystability in this case. However, it turns out that when
$\Delta_+$ is equipoised, it is automatically K-polystable: for $q(z)=1$,
$\deg(A+B)\leq 1$, and so between any maximum of $A$ on $(\al_\1,\al_\2)$ and
$B$ on $(\be_\1,\be_\2)$, the quadratic $A''=-B''$ has a unique root; the
boundary conditions thus force $A$ and $B$ to be positive on $(\al_\1,\al_\2)$
and $(\be_\1,\be_\2)$ respectively.

For equipoised trapezia (which cannot be orthotoric), Legendre~\cite{Eveline}
established similar existence and K-polystability results using ambitoric
metrics of Calabi type.
\end{ex}

We now generalize these results to arbitrary quadrilaterals, on which we
relate the existence of ambitoric extremal K\"ahler metrics to the toric
K-polystability criteria.

\begin{thm} \label{thm:stability} Let $(M,\omega,\T)$ be a toric symplectic
orbifold whose rational Delzant polytope $\Delta$ is a quadrilateral
\textup(i.e., $b_2(M)=2$\textup).  Then the following are equivalent\textup:
\begin{numlist} 
\item $(M,\omega)$ admits a $\T$-invariant extremal K\"ahler metric\textup;
\item $(M,\omega,\T)$ is analytically relatively K-polystable wrt.~toric
degenerations\textup;
\item $(M,\omega)$ admits a $\T$-invariant ambitoric extremal K\"ahler metric
$g$ which is regular on generic orbits, unless $\Delta$ is an equipoised
trapezium, in which case, $g$ has Calabi type or is a K\"ahler product.
\end{numlist}
In particular, if $(M,\omega,\T)$ admits an extremal K\"ahler metric, it must
be ambitoric.
\end{thm}
\begin{proof} We use the notation of sections~\ref{s:toric-symp}
and~\ref{s:quad}, so that $\Delta\subset\Xi\subset {\torh}^*$, where
$\iota\transp(\Xi)=\{1\}$ for an affine structure $\iota\colon\R\to\torh$ on
$\Proj(\torh^*)$ which we identify with $\iota(1)\in\torh$. Since $\Delta$ is
convex, $\iota(1)$ is interior to the strictly convex cone spanned by the
normal rays of $\Delta$; thus $[\iota(1)]$ is interior to the dual polytope
$\Delta^*\subset \Proj(\torh)$ which is the projectivization of this cone.
Let $\eaf\in\torh$ be the extremal affine function and $\cC(\Delta,\Lab)$ the
unique conic in $\Proj(\torh)$ passing through the normals, and such that
$[\iota(1)]$ and $[\eaf]$ are orthogonal.

\smallbreak\noindent{\bf Case 1}. Suppose $\Delta$ is temperate and
$\cC(\Delta,\Lab)$ is nonsingular.  Then by Lemma~\ref{l:cond1},
$\cC(\Delta,\Lab)$ satisfies Condition~\ref{conic-meets} and so
Proposition~\ref{p:existence} implies that there are positive ambitoric
compactifications with rational Delzant polytope $\Delta$ and induced conic
$\cC(\Delta,\Lab)$.

Fixing $\cC=\cC(\Delta,\Lab)$ and the associated factorization structure, such
compactifications also exist for arbitrary positive rational rescalings of the
normals of $\Delta$. Hence we are in a position to apply an argument pioneered
by E.~Legendre~\cite{Eveline} in the parabolic case.  The positive ambitoric
Ansatz, with fixed $\al_k,\be_k$ and $q$ yields a linear condition on the
normal parameters $r_{\al,k},r_{\be,k}$ for the existence of quartics $A,B$
satisfying the boundary conditions~\eqref{AB-boundary} such that
$A(z)+B(z)=q(z)\pi(z)$ with $\pi$ orthogonal to $q$.  If such $A$ and $B$
exist, then even if they do not satisfy the positivity requirement to define
an extremal K\"ahler metric, we can use them to compute that the extremal
affine function is orthogonal to $q$. However, this latter condition is also a
linear condition on the normal parameters $r_{\al,k},r_{\be,k}$. We conclude
that the two linear conditions agree. For the normals $\Lab$ of $\Delta$,
$\eaf$ is orthogonal to $q$; hence there do exist quartics $A, B$ defining a
formal extremal solution $\mathbf H=\mathbf H^+$ on $\Delta=\Delta_+$.

\smallbreak\noindent{\bf Case 2}. Suppose $\Delta$ is temperate, but
$\cC(\Delta,\Lab)$ is singular so that it is the line-pair in the pencil of
conics through the four normals which meets the interior of $\Delta^*$. If
$[\iota(1)]$ is on $\cC(\Delta,\Lab)$, then $\Delta$ is an equipoised
trapezium, hence admits formal extremal solution of Calabi
type~\cite{Eveline}.  Otherwise, Proposition~\ref{p:existence} implies that
there are negative ambitoric compactifications with rational Delzant polytope
$\Delta$. As in step 1, we conclude that there are quartics $A,B$ defining a
formal extremal solution $\mathbf H=\mathbf H^-$ on $\Delta=\Delta_-$.

\smallbreak\noindent{\bf Case 3}. If $\Delta$ is intemperate, it is not
K-polystable by Lemma~\ref{l:unstable}. Otherwise, either step 1 or step 2
provides a formal extremal solution. This may not yield a positive definite
metric, but it can be used to compute the toric K-polystability criterion. In
the Calabi or product case, this has been done in~\cite{Eveline}; it remains
to consider in the regular ambitoric case.

Let $\mathbf H$ be the formal extremal solution given by the quartics $A, B$
as above. If $A$ is positive on $(\al_\1,\al_\2)$ and $B$ is positive on
$(\be_\1,\be_\2)$, then $\mathbf H$ is positive definite.
Hence~\eqref{Hfutaki} implies that $\cF_{\Delta,\Lab}(f)>0$ for $f\in
\mathcal{PL}(\Delta)$, unless $\Hess f=0$ on $\Delta$, i.e., $f$ is
affine. Hence the rational Delzant polytope is K-polystable.

To establish a converse, we consider special families of simple convex PL
functions determined by the ambitoric factorization. For any $x_0 \in
(\al_\1,\al_\2)$ consider the line segment $\{(x_0,y) : y \in
(\be_\1,\be_\2)\}$.  Under $\mu^\pm$ it transforms to a line segment $S_{x_0}$
in the interior of $\Delta_\pm$.  Let $u_{x_0}$ be a normal of $S_{x_0}$. It
is straightforward to check that ${\mathbf H}^\pm_{(x_0,y)}(u_{x_0},u_{x_0})$
is positive multiple of $A(x_0)$. Thus, if $(M,\omega)$ is analytically
relatively K-polystable with respect to toric degenerations,
then~\eqref{Sfutaki} implies $A(x_0)$ must be positive for any $x_0 \in
(\al_\1,\al_\2)$; the argument for $B(z)$ is similar.

\smallbreak\noindent{\bf Conclusion}. We conclude that (ii) and (iii) are
equivalent, and evidently (iii) implies (i).  The implication (i)
$\Rightarrow$ (ii) follows from~\cite[Theorem~1.3]{ZZ2}, and the final
assertion follows from the uniqueness of the extremal toric K\"ahler metrics,
modulo automorphisms, established in~\cite{Guan}.
\end{proof}

\begin{rem}  The general theory from \cite{Donaldson1} and \cite{ZZ1} implies
that in order to check the K-polystability of a rational Delzant polygon
$(\Delta,\Lab)$, it is only necessary to consider a particular kind of PL
convex function: the simple PL convex functions whose crease meets the
interior of the polytope $\Delta$. Theorem~\ref{thm:stability} shows that in
the case of a quadrilateral, it suffices to consider the cases that the crease
is either one of the diagonals or meets the polytope in a segment
corresponding to $\{ (x_0, y) : y \in (\beta_0, \beta_{\infty}) \}$ or
$\{(x,y_0) : x \in (\alpha_0, \alpha_{\infty})\}$ under the unique ambitoric
compactification given by the conic $\mathcal{C}(\Delta,\Lab)$, which may be
found by solving linear equations.

In the light of \cite{Donaldson1} and its extension to orbifolds
in~\cite{RT2}, when the rational Delzant polytope $\Delta$ has rational
vertices with respect to the dual lattice, one can also consider a weaker
version of \emph{algebraic} relative (toric) K-polystability by requiring that
$\cF_{\Delta,\Lab}(f) \ge 0$ for any \emph{rational} PL continuous convex
function $f$ with equality if and only if $f$ is an affine function.
Presumably, this condition corresponds to an algebro-geometric notion of
stability for the corresponding (log) variety.  A key observation in
\cite{Donaldson1} is that in the case of a rational polygon with vanishing
extremal vector field, the algebraic relative K-polystability with respect to
toric degenerations is equivalent to the analytic one. This phenomenon is well
demonstrated on our classification: if $\al_k, \be_k, r_{\al,k}, r_{\be,k}$
are all rational numbers as in Remark~\ref{rational-polytope} (so that the
vertices of $\Delta$ are rational with respect to the dual lattice) and if
$\cF_{\Delta,\Lab} > 0$ on rational PL convex functions which are not affine
on $\Delta$, we then conclude as in the proof of Theorem~\ref{thm:stability}
that $A(z)$ must be positive at any \emph{rational} point in $(\al_\1,
\al_\2)$. It follows that $A(z) \ge 0$ on $(\al_\1, \al_\2)$ with (possibly) a
repeated irrational root in this interval.  As the $\al_k$'s and $r_{\al,k}$'s
are rational, by the first order boundary conditions $A(z)$ is a (multiple of)
degree $4$ polynomial with rational coefficients with two simple (rational)
roots $\al_\1$ and $\al_\2$. In particular, any double root of $A$ (if any)
must be rational too, showing that $A(z)$ must be strictly positive on
$(\al_\1, \al_\2)$.  Similarly, $B(z)>0$ on $(\be_\1, \be_\2)$.

This provides a computational test for K-polystability of quadrilaterals which
is guaranteed to terminate in the unstable case. We will further use these
observations in Appendix~\ref{s:semistablity} to show that any compact convex
quadrilateral which is not a parallelogram can be made K-unstable by a suitable
choice of the affine normals $\Lab$.
\end{rem}
\begin{rem} In view of Remark~\ref{zero-measure}, the equivalence (ii)
$\Leftrightarrow$ (iii) of Theorem~\ref{thm:stability} extends to complements
of toric divisors in compact toric orbifolds (with $b_2=2$), for ambitoric
extremal K\"ahler metrics satisfying~\eqref{AB-boundary} with $r_{\al,k}$ or
$r_{\be,k}$ zero on omitted facets.
\end{rem}

\section{Examples}\label{s:examples}

Our results show, as in~\cite{Eveline}, that for any convex quadrilateral,
there is a nonempty open subset of scales for the normals such that the
corresponding toric $4$-orbifold has an extremal K\"ahler metric. By
considering data close to well-known Bochner-flat K\"ahler metrics, we shall
demonstrate this explicitly. We shall restrict attention to rational data in
the sense of Remark~\ref{rational-polytope}. More precisely, if $\al_k,\be_k$
and the coefficients of $q$ are rational, then the parameters $\dd$ and $\ad$
defining the quadrilateral are rational, and the normal scales $r_{\al,k}$,
$r_{\be,k}$ are constrained by a single rational linear relation.

A $4$-dimensional extremal K\"ahler metric with nonzero scalar curvature is
locally conformally Einstein iff it is Bach-flat, and globally so if the
scalar curvature is nonvanishing~\cite{De}; the compact smooth examples have
been classified~\cite{CLW,LeBrun1,LeBrun2a}, so we seek complete or compact
orbifold examples.

A $4$-dimensional K\"ahler metric is Bochner-flat iff it is selfdual
($W_-=0$); hence it is Bach-flat and locally conformally Einstein.  According
to R. Bryant~\cite{Bryant} such metrics exist on weighted projective planes
$\C P^2_{\wt_1,\wt_2,\wt_3}$ (where $\wt_1,\wt_2,\wt_3$ are positive integers
with no common factor), cf.~section~\ref{s:simplex}
and~\cite{Abreu1,ACGT}. Since $W^-=0$, there is some freedom in the choice of
negative complex structure $J_-$, and hence a family of ambitoric structures
compatible with a given Bochner-flat K\"ahler metric (note however, that $J_-$
is not globally defined).

\subsection{Bochner-flat ambitoric structures on weighted projective planes}
\label{s:weighted-proj-plane}

The K\"ahler metric $(g_+,J_+,\omega_+)$ of an ambitoric structure is
Bochner-flat if $A(z)=P(z)$, $B(z)=-P(z)$ for an arbitrary polynomial $P$ of
degree $\le 4$.  The parabolic case (with $q(z)=1$) has been studied
in~\cite{ACGT}: we now consider arbitrary $q$.

We set $P(z)=-(z-z_0)(z-z_1)(z-z_2)(z-z_3)$, where $z_0 < z_1 < z_2 < z_3$ and
$q(x,y)$ is positive on $[z_2,z_3]\times[z_1,z_2]$ (e.g., $q(z)$ positive on
$[z_1,z_3]$). Since $\Delta_+$ is a simplex, we are in the degenerate case of
section~\ref{s:ambi-orbifolds} where $\al_\1=\be_\2$, and we set $\be_\1=z_1$,
$\be_\2=z_2=\al_\1$ and $\al_\2=z_3$. The boundary conditions give
\begin{equation*}
r_{\be,\1}=-2P'(z_1),\quad r_{\be,\2}=r_{\al,\1}=-2P'(z_2), \quad
r_{\al,\2}=-2P'(z_3),
\end{equation*}
which we can ensure are rational by taking $z_0,\ldots z_3$ rational. By
Remark~\ref{rational-polytope}, taking $q$ also rational gives
condition~\eqref{rational} for the normals $u_1:=\p{\be_\1}/r_{\be,1}$,
$u_2:=\p{\al_\1}/r_{\al,1}$ and $u_3:=\p{\al_\2}/r_{\al,2}$. These normals are
$u_j =-\p{z_j}/2P'(z_j)$, which satisfy
\begin{equation*}
(z_1-z_0)q(z_2,z_3)u_1+(z_2-z_0)q(z_1,z_3)u_2+(z_3-z_0)q(z_1,z_2)u_3=0
\end{equation*}
and so the weights $\wt_1,\wt_2,\wt_3$ are a multiple of
$(z_1-z_0)q(z_2,z_3)$, $(z_2-z_0)q(z_1,z_3)$, $(z_3-z_0)q(z_1,z_2)$.  Any
weighted projective plane $\C P^2_{\wt_1,\wt_2,\wt_3}$ with distinct weights
has a Bochner-flat ambitoric structure of any type (parabolic, hyperbolic or
elliptic).

Since the scalar curvature $s_+$ of $g_+$ is an affine function of the
momenta, it attains its maximum and minimum values at the vertices of the
momentum simplex, which are the images of $(x,y)=(z_2,z_1)$, $(z_3,z_1)$ and
$(z_3,z_2)$. If we write, for $0< i<j\leq 3$, $P(z)=-(z-z_i)(z-z_j)p_{ij}(z)$,
then we compute from section~\ref{s:extremal} that
\begin{equation*}
s_+(z_j,z_i) = 3\frac{q(z_i)p_{ij}(z_j)-q(z_j)p_{ij}(z_i)}{z_j-z_i}.
\end{equation*}
We deduce (assuming $q(z)>0$ on $[z_1,z_3]$) that $s_+$ is positive at
$(z_3,z_1)$; it is also positive at $(z_3,z_2)$ for $z_2-z_1$ sufficiently
small.  On the other hand, at $(z_2,z_1)$, for $z_3-z_2$ small, $s_+$ changes
sign as a function of $z_0\in (-\infty,z_1)$, being negative at $z_0=z_1$, but
positive once $z_1-z_0$ is sufficiently large.

Under these conditions, $s_+$ is everywhere positive for $z_0\ll z_1$, and
hence $g_+$ is globally conformal to an Einstein hermitian metric $s_+^{-2}g$
of positive scalar curvature. On the other hand, as $z_0$ increases, $s_+$
becomes nonpositive on the preimage $(\mu^+)^{-1}(T)$ of a triangle
$T\subset\Delta_+$ containing the vertex $(z_2,z_1)$.  This preimage $N$ a
compact orbifold with boundary (the latter being the zero locus of $s_+$), but
it is straightforward to see that $N$ is covered by a compact manifold $\tilde
N$ with boundary, cf.~\cite{Bryant}.  Indeed let $C$ be the $2$-dimensional
cone defined by the two facets of $\Delta_+$ which bound $T$ and let $\Lambda$
be the lattice generated by the normals to the these facets; Delzant theory
identifies $(C, \Lambda)$ as the image by the momentum map of a (standard)
toric $\C^2$; the preimage $\tilde N$ of $T\subset C$ is the closure of a
bounded domain biholomorphic to the unit ball in $\C^2$.  The lift of
$g=s_+^{-2}g_+$ to $\tilde N\setminus \del \tilde N$ is a conformally compact,
Einstein hermitian metric of negative scalar curvature, which is complete
since $ds_+ \neq 0$ on $\del\tilde N$ (cf.~\cite{Anderson}).

\subsection{Extremal ambitoric compactifications}\label{s:ext-comp}

In order to obtain new examples, which have those of the previous subsection
as limiting cases, we let $P(z)$, $q(z)$ and $z_0<z_1<z_2<z_3$ be as before,
and consider rational $\al_k$ and $\be_k$ satisfying $z_1\approx
\be_\1<\be_\2\lesssim z_2\lesssim \al_\1<\al_\2\approx z_3$.  We now set
$A(z)=q(z)\pi_A(z)+P(z)$ and $B(z)=q(z)\pi_B(z)-P(z)$ where $\pi_A$ and
$\pi_B$ are quadratic polynomials uniquely determined by three rational
(affine) linear conditions: each is orthogonal to $q$, $A(\al_\1)=0=A(\al_\2)$
and $B(\be_\1)=0=B(\be_\2)$.  Note that $A(z)+B(z)=q(z)(\pi_A+\pi_B)(z)$ and
that $A(z)-B(z)=q(z)(\pi_A-\pi_B)(z) +2P(z)$.

For $\be_\1=z_1$, $\be_\2=z_2=\al_\1$, $\al_\2=z_3$, the unique solution is
$\pi_A=\pi_B=0$ and the quartics $A$ on $(z_3,z_2)$ and $B$ on $(z_2,z_1)$ are
positive and define a Bochner-flat extremal metric. Hence for a small
perturbation of the endpoints, $A$ and $B$ remain positive on
$(\al_\1,\al_\2)$ and $(\be_\1,\be_\2)$ respectively (having roots close to
$z_1,z_2,z_3$ and $z_0<z_1$). The boundary conditions $r_{\al,k}=-2A'(\al_k)$,
$r_{\al,k}=2B'(\be_k)$ give rational scales for the normals with the right
signs to obtain an extremal K\"ahler metrics over a rational Delzant
quadrilateral $\Delta_+$. Since $\be_\2$ and $\al_\1$ are very close, the
sides $\Fa_{\be,\2}$ and $\Fa_{\al,\1}$ are almost parallel, meaning that the
quadrilateral $\Delta=\Delta_+$ has parameter $\ad$ close to $-1$, but $\dd\in
(-1,1)\cap\Q$ is unconstrained.

The parametrization of these solutions by $P,\al_k,\be_k$ is not effective,
because $P(z)$ is only determined up to the addition of $q(z)\pi(z)$ with
$\pi$ orthogonal to $q$. This overcounting matches precisely with the
dimension of the space of rational Delzant quadrilaterals. By symmetry, we see
that for quadrilaterals with rational parameters $\dd$ and $\ad$, one of these
being sufficiently close to $\pm1$, there is a nonempty open subset of
rational normal scales---belonging, up to homothety, a nonempty open subset of
$\Q P^3$---for which the corresponding toric $4$-orbifold has an extremal
K\"ahler metric. There are thus infinitely many ambitoric extremal compact
$4$-orbifolds with $b_2=2$, depending on $5$ rational parameters.

\subsection{Conformally Einstein K\"ahler orbifolds and complete Einstein
metrics}\label{s:einstein}

A regular extremal ambitoric structure, given by quartic polynomials
$A=q\pi+P,B=q\pi-P$ is Bach-flat iff the quadratics $\pi$ and $C_q(P)$ (which
are both orthogonal to $q$) are linearly dependent. For fixed $q$, this is a
singular quadric hypersurface in the $\Q P^6$ of coefficients of $(\pi,P)$ up
to homothety.  Bochner-flat metrics are Bach-flat with $\pi=0$, and so the
quadric meets any open neighbourhood of $\pi=0$. Hence, as in the Bochner-flat
case, we obtain locally or globally conformally Einstein metrics according to
whether the scalar curvature $s_+$ of $g_+$ changes sign or is positive.

We can make this more explicit using the approach developed in the previous
two subsections, where $A=P$, $B=-P$ gives a known Bochner-flat K\"ahler
metric with nonzero scalar curvature.  Fix $\pi,\tilde\pi:=C_q(P)$ as a basis
for the quadratic polynomials orthogonal to $q$, and consider the equations
$A+B = \delta(\lambda\tilde\pi+ \mu\pi)q$, $C_q(A-B)=\gamma
(\lambda\tilde\pi+\mu \pi)$, $A(\al_k)=0=B(\be_k)$. For fixed
$(\delta,\gamma)\approx(0,1)$ and $z_0\approx \be_\1<\be_\2\lesssim
z_1\lesssim \al_\1<\al_\2\approx z_2$, this has a unique (and appropriately
positive) solution up to scale (with $\lambda\approx 2$ and $\mu\approx
0$). The solution depends rationally on $[\delta:\gamma]$ up to scale, hence
for given $\al_k,\be_k$, we have a one parameter family of Bach-flat ambitoric
orbifolds.

Positivity of $s_+$ can be obtained by a limiting argument, provided $P$ is
chosen so that the corresponding Bochner-flat metric
(in~\ref{s:weighted-proj-plane}) has positive scalar curvature. We thus obtain
infinitely many new examples of compact ambihermitian Einstein $4$-orbifolds
of positive scalar curvature. If instead we choose $P$ so that the
corresponding Bochner-flat metric has scalar curvature positive for
$(x,y)=(z_3,z_1), (z_3,z_2), (z_2,z_2)$, and negative for $(x,y)=(z_2,z_1)$,
the analysis in~\ref{s:weighted-proj-plane} generalizes to yield new complete
ambihermitian Einstein $4$-manifolds of negative scalar curvature.

We do not attempt to classify explicitly the data yielding Bach-flat (or
extremal) ambitoric compactifications, but examples are not confined to
limiting cases. For instance, let $q(z)=z$ and consider the quartics $A,B$
with parameters $(s,t)$ given by
\begin{align*}
A(z)&= tz^4+(s-1)(t+1)z^3-(st+4s+2t-2)z^2-2s(t-2)z\\
&=z(z-2)(tz^2+(st+s+t-1)z+s(t-2))\\
B(z)&= -tz^4+(s-1)(t-2)z^3+(st+4s+2t-2)z^2-2s(t+1)z\\
&=-z(z-1)(tz^2-(st-2s-2t+2)z-2s(t+1))
\end{align*}

In this family, the roots $z=0,1,2$ are fixed, which is a slightly special
choice because $q(z)$ vanishes at $z=0$, and so $(A+B)(0)=0$ is a consequence
of the extremality condition. The latter equation is satisfied by the family,
since
\begin{equation*}
(A+B)(z) = z(2t-1)\bigl((s-1)z^2-2s\bigr)
\end{equation*}
and the Bochner-flat case is $t=1/2$, with $A(z)=-B(z)=\frac12
z(z-1)(z-2)(z+3s)$.  With three roots fixed, the extremal family is
parametrized by an open subset of $\Q P^3$, and $s,t$ are affine coordinates
on the quadric surface given by the Bach-flatness condition $a_1a_3=b_1b_3$ on
the coefficients of $A$ and $B$.

For $s>0$, after negating $A,B$, we are in the situation considered before,
with $z_0=-3s$, $z_1=0$, $z_2=1$ and $z_3=2$: the Bochner-flat metric with
$\be_\1=0,\be_\2=1$ and $\al_\2=2$ has positive scalar curvature.  Varying $t$
in $[2/3(s+2),1/2]$, $A$ has a root $1<\al_\1<2$, yielding Bach-flat examples
over $[\al_\1,2]\times[0,1]$. We also get such examples for $s<-2/3$ using a
slight variant of the same approach in which $z_0=-3s>2$.  Here $A,B$
(unnegated) satisfy positivity on $[\al_\1,2]\times[0,1]$ with $1<\al_\1<2$
provided $-1<s<-2/3$ and $1/2<t<2/3(s+2)$, or $s<-1$ and $1/2>t>2/3(1-s)$.
The quadrilaterals corresponding to these examples have moduli $\dd=1/2$ and
$\ad=-1/\al_\1$.

Similar examples to these can be found by considering the Bochner-flat metrics
on $[2,-3s]\times [1,2]$ for $s<-2/3$. However, there are plenty of examples
which are not deformations of the Bochner-flat family in this way. Other
convenient families are given by $(s,t)$ coordinate lines tangent to the
discriminant of $A$ or $B$, so that one of the quartics splits over $\Q(s)$ or
$\Q(t)$. These lines are $s=0,-1,1/3,2/3,2,\infty$ and $t=0,-1,1/2,2,\infty$.
Many of these only yield singular or indefinite examples. However, after
multiplying $A$ and $B$ by $-1/t$, we have, for $s=1/3$, $t=2/(1+3u)$,
\begin{align*}
A(z)&= -z^4 + (u+1) z^3 -(u-2) z^2- 2 u z = -(z+1)z(z-2)(z-u)\\
B(z)&= z^4- 2u z^3 +(u-2) z^2 +(u+1) z = z(z-1)\bigl(z^2-(2u-1)z-(u+1)\bigr),
\end{align*}
and for $u>2$, $A(z)$ is positive on $(2,u)$ while the nontrivial roots of
$B(z)$ have opposite sign and sum at least $3$, so that $B(z)$ is positive on
$(0,1)$. Hence after rescaling, we obtain Bach-flat examples on
$[2,\al_\2]\times[0,1]$ with $\al_\2=u>2$. The quadrilaterals corresponding to
these examples have moduli $\dd=1/\al_\2$ and $\ad=-1/2$.

For a final example, let $t=0$, scale by $-1/(s-1)$ and set $s=u/(u-1)$ so
that
\begin{align*}
A(z)&= -z^3 + 2(u+1) z^2 - 4 u z = -z(z-2)(z-2u)\\
B(z)&= 2z^3 - 2(u+1) z^2 + 2 u z = 2z(z-1)(z-u).
\end{align*}
For $\frac12<u<1$ this yields Bach-flat examples on
$[\al_\1,2]\times[0,\be_\2]$ with $\be_\2=u$ and $\al_\1=2u$, while for $u>1$
we obtain instead examples on $[2,\al_\2]\times [0,1]$ with $\al_\2=2u$.

\subsection{Hirzebruch orbifold surfaces}

Another interesting class of examples are the toric orbifolds for which the
rational Delzant polytope is a trapezium but not a parallelogram.  It is shown
in~\cite{Eveline} that these are precisely the toric orbifolds which admit
toric K\"ahler metrics of Calabi type. Up to an orbifold covering, these
orbifolds are fibre bundles of the form $M = P\times_{S^1} \C P_{w_1,w_2} \to
\C P_{v_1,v_2}$, where the fibre and the base are weighted projective lines
$\C P_{w_1, w_2}$ and $\C P_{v_1, v_2}$ respectively, and $P$ is a principal
$S^1$-orbibundle over $\C P_{v_1, v_2}$. It follows from~\cite{Eveline} that
such a Hirzebrich orbifold surface admits an extremal K\"ahler metric of
Calabi type (in some and hence any K\"ahler class) if and only if the base
admits a CSC K\"ahler metric, i.e. $v_1=v_2=1$. In our formalism, this
corresponds to the case when $\Delta$ is an equipoised trapezium, and
$(M,\omega, \T)$ is automatically K-polystable with respect to toric
degenerations~\cite{Eveline}.

When $v_1 \neq v_2$, the corresponding trapezia are not equipoised and
extremal K\"ahler metrics must be obtained from the hyperbolic ambitoric
ansatz.  Rational Delzant trapezia which are close to but different from
equipoised ones provide such examples. On the other hand, one can readily find
K-unstable trapezia by violating the condition $(1+|\dd|)(1+|\ad|)<2$ in
Lemma~\ref{l:unstable}; then there exist affine normals such that the
trapezium is intemperate. More generally, Proposition~\ref{unstable} shows
that any quadrilateral which is not a parallelogram is K-unstable for some
choice of affine normals.

\appendix

\section{Factorization structures}\label{s:fs}

The idea behind factorization structures is to separate variables using a
rational map from a product of projective lines to projective space. In order
to explain our terminology, and place our constructions in a natural context,
we discuss this idea in greater generality than we need in the body of the
paper.

\subsection{Factorization for rational Delzant polytopes}

Let $\torh$ be a real vector space of dimension $m+1$ and
$\Delta\sub\Proj(\torh^*)$ the image of a strictly convex cone in $\torh^*$.

\begin{defn} A \emph{factorization structure} over $\Proj(\torh^*)$ is an
injective linear map $\fs\colon \torh\to V_1^*\otimes V_2^*\otimes\cdots\otimes
V_m^*$, where $V_1,\ldots V_m$ are $2$-dimensional real vector spaces, such
that composite $S_\fs$ of the \emph{Segre embedding}
\begin{align*}
\Proj(V_1)\times\cdots\times\Proj(V_m)&\to \Proj(V_1\otimes\cdots\otimes V_m)\\
([v_1],\ldots [v_m])&\mapsto [v_1\otimes\cdots\otimes v_m]
\end{align*}
with the dual projection $\Proj(V_1\otimes\cdots\otimes V_m) \dashrightarrow
\Proj(\torh^*)$ maps any coordinate hyperplane ($[v_j]$ constant for some $j$)
into a hyperplane in $\Proj(\torh^*)$. We say $\fs$ is \emph{compatible} with
$\Delta$ (or a factorization structure \emph{for} $\Delta$) if $S_\fs$ maps a
product $I_1\times \cdots\times I_m$ of intervals $I_j\sub\Proj(V_j)$
bijectively onto $\Delta\sub\Proj(\torh^*)$.
\end{defn}
Note that the coordinate hyperplane condition is automatic for $m\leq 2$
(since $S_\fs$ maps coordinate lines to lines). Also $S_\fs$ maps the boundary
of $I_1\times \cdots\times I_m$ to the boundary of $\Delta$, so $\Delta$ has
at most $2m$ facets. In our application, $\Delta$ and the intervals $I_j$
will be closed, so $S_\fs$ is also a bijection between boundaries.

If $\fs$ is understood, we typically regard it as an inclusion and identify
$\torh$ with its image $\fs(\torh)$ in $V_1^*\otimes\cdots\otimes V_m^*$. The
examples we consider are all of the following form.

\begin{exs} Let $(m_1,\ldots m_k)$ be a partition of $m$ and let $W_1,\ldots
W_k$ be $2$-dimensional vector spaces. Then $\fs\colon \torh\to
\bigotimes_{i=1}^k \bigl(\otimes^{m_i} W_i^*\bigr)$ is a
\emph{Segre--Veronese} factorization structure of \emph{type} $(m_1,\ldots
m_k)$ iff $\fs(\torh)=\sum_{i=1}^k \spn{\bp_1}\otimes\cdots \otimes
S^{m_i}W_i^*\otimes \cdots \otimes \spn{\bp_k} \sub \bigotimes_{i=1}^k S^{m_i}
W_i^*$ for some decomposable $\bp_i={\gavec_i}^{\odot m_i}\in S^{m_i} W_i^*$
(for $j=1\ldots k$, $\gavec_j\in W_j^*$).

The map $S_\fs\colon \Proj(W_1)^{m_1}\times\cdots\times\Proj(W_k)^{m_k}\to
\Proj(\torh^*)$ sends a coordinate hyperplane with one component equal to
$[\alvec_j]\in \Proj(W_j)$ to the hyperplane in $\Proj(\torh^*)$ dual to
$[\bp_1\otimes\cdots \otimes {\thvec_j}^{\odot m_j} \otimes\cdots \otimes
  \bp_k]\in \Proj(\torh)$, where $\ker\thvec_j=\spn{\alvec_j}$. This is an
element in the image of the (dual) mixed \emph{Segre--Veronese embedding}
$\Proj(W_1^*)\times\cdots\times\Proj(W_k^*)\to \Proj\bigl(\bigotimes_{i=1}^k
S^{m_i} W_i^*\bigr)$.
\end{exs}

The extreme partitions $(1,1,\ldots 1)$ and $(m)$ correspond to pure Segre and
Veronese embeddings respectively. For toric $4$-orbifolds $(m=2)$, these are
the only cases.

\subsection{Factorizations on toric $4$-orbifolds}\label{s:fs4}

When $m:=\dim\tor=2$, the image of any factorization structure $\torh\to
W_1^*\otimes W_2^*$ ($\dim W_i=2$) is the annihilator of an element $\chi$ of
$W_1\otimes W_2$. If $\chi=\gavec_1\otimes \gavec_2$ is decomposable, the
image of $\torh$ is ${\gavec_1}^0\otimes W_2^* + W_1^* \otimes {\gavec_2}^0$,
where ${\gavec_j}^0\sub W_j^*$ is the annihilator of $\gavec_j\in W_j$. If
not, $\chi$ defines an isomorphism $W_1^*\to W_2$, and hence, fixing a nonzero
area form on $W_1$, an isomorphism $W_1\to W_2$. Using this to identify $W_1$
with $W_2$ and dropping subscripts, the factorization structure $\torh \to
W^*\otimes W^*$ has image annihilating $\Wedge^2W\subset W\otimes W$, i.e.,
equal to $S^2W^*$.

Thus, up to isomorphism, any factorization structure is Segre--Veronese of
type $(1,1)$ or $(2)$; these are the Segre and Veronese factorizations used in
the paper.  In the Segre case, $\Proj(W_1)\times\Proj(W_2) \dashrightarrow
\Proj(\torh^*)$ is projection away from the point $[\gavec_1\otimes\gavec_2]$
on the quadric surface in $\Proj(W_1\otimes W_2)$; this is the famous
birational map identifying the blow-up of $\Proj(W_1)\times\Proj(W_2)$ at
$([\gavec_1],[\gavec_2])$ with the blow-up of $\Proj(\torh^*)$ at two points.
In the Veronese case, the map $\Proj(W)\times\Proj(W)\dashrightarrow
\Proj(\torh^*)$ is projection away from a point \emph{off} the quadric
surface, which is a branched double cover over a conic.

\section{The semistability surface}\label{s:semistablity}

In this appendix we consider the dependence of the toric K-polystability
condition (and hence the existence of extremal metrics) on the rational
Delzant quadrilateral $(\Delta,\Lab)$, which is determined by a positivity
property of its Futaki functional $\cF_{\Delta,\Lab}$ on the space
$\mathcal{PL}(\Delta)$ of PL convex functions. For a fixed quadrilateral
$\Delta$, $\cF_{\Delta,\Lab}$ depends \emph{linearly} on inverse scales
$r_{\al,k}$ and $r_{\be,k}$ ($k=\1,\2$) for the normals $\Lab$. We can thus
parameterize a choice of normals, up to overall scale, by a point
$[r_{\al,\1},r_{\al,\2}, r_{\be,\1},r_{\be,\2}]$ in the positive quadrant of
$\Q\Proj^3$, and a given choice will be K-polystable provided this point lies
in the open subset $R_\Delta$ of $\R\Proj^3$ on which the Futaki functional
has constant sign (with only trivial zeros). It follows from~\cite{Eveline}
that $R_\Delta\sub \R\Proj^3$ has nonempty intersection with the positive
rational quadrant.

We refer to the boundary $S_\Delta$ of $R_\Delta$ as the \emph{semistability
  surface} of $\Delta$. At any point in $S_\Delta$ there must be a nontrivial
Futaki invariant which is zero, and since it is linear on $\R\Proj^3$, this
Futaki invariant defines a supporting hyperplane for $R_\Delta$. Consequently,
we can hope to describe the \emph{dual surface} of $S_\Delta$ explicitly in
terms of Futaki invariants, and then consider its dependence on $\Delta$.

It suffices to consider Futaki invariants defined by simple PL convex
functions with a crease meeting opposite sides of $\Delta$ (including the
diagonals of $\Delta$ as extreme cases): our main results show that the
positivity of these invariants is not only necessary, but sufficient, for
toric K-polystability. These invariants are still quite formidable in
complexity, but are amenable to computation.

In our computations, we drop overall positive constants, such as the constant
$c(\dd,\ad)=24/\bigl(4-(1-\dd^2)(1-\ad^2)\bigr)$ appearing in the extremal
affine function, and employ the dihedral symmetry (which acts projectively on
$\Delta$) to minimize duplication of effort. This symmetry group, determined
by its action on vertices, is generated by a ``vertical'' reflection
(cf. Figure 1) $\sigma_\al\colon v_{\1\1}\mapsto v_{\2\1}, v_{\1\2}\mapsto
v_{\2\2}$ and a diagonal reflection $\sigma_\dd\colon v_{\1\1}\mapsto
v_{\2\2}$ fixing $v_{\1\2}$ and $v_{\2\1}$, so that
$\rho:=\sigma_\dd\circ\sigma_\al$ is a $\frac{\pi}{2}$ rotation, which acts on
vertices and edges by
\begin{gather*}
v_{\1\1}\mapsto v_{\2\1} \mapsto v_{\2\2} \mapsto v_{\1\2} \mapsto v_{\1\1},\\
\Fa_{\al,\1}\mapsto \Fa_{\be,\1}\mapsto \Fa_{\al,\2}\mapsto \Fa_{\be,\2}
\mapsto \Fa_{\al,\1}.
\end{gather*}
The remaining nonidentity elements consist of the other diagonal reflection
$\sigma_\ad=\sigma_\al\rho=\sigma_\al\sigma_\dd\sigma_\al$, the ``horizontal''
reflection $\sigma_\be=\rho\sigma_\dd=\sigma_\dd\sigma_\al\sigma_\dd$,
$\rho^2=\sigma_{\ad}\sigma_{\dd}$ and $\rho^3=\rho^{-1}$. The dihedral action
is only affine after permuting the labelling, so there is an induced action on
the parameters $(\dd,\ad)$ which determine the affine class of $\Delta$ as a
labelled quadrilateral. Explicitly, we have
${\sigma_\al}^*(\dd,\ad)=(\ad,\dd)$ and ${\sigma_\dd}^*(\dd,\ad)= (-\dd,\ad)$,
and hence $\rho^*(\dd,\ad)=(\ad,-\dd)$, ${\sigma_\be}^*(\dd,\ad)=(-\ad,-\dd)$,
${\sigma_\ad}^*(\dd,\ad)=(\dd,-\ad)$.

The two families of simple PL convex functions whose Futaki invariants we need
are $f^\al_{\bs,\bt}$, with a crease joining $\bs\in \Fa_{\al,\1}$ to $\bt\in
\Fa_{\al,\2}$, and $f^\be_{\bs,\bt}$, with a crease joining the $\bs\in
\Fa_{\be,\1}$ to $\bt\in \Fa_{\be,\2}$. We write
\begin{equation*}
\cF _{\Delta,\Lab} (f^\al_{\bs,\bt})
= \sum_{j\in \{\al,\be\}\times\{\1,\2\}} A_j (\dd,\ad,\bs,\bt) r_j,\quad
\cF _{\Delta,\Lab} (f^\be_{\bs,\bt}) 
= \sum_{j\in \{\al,\be\}\times\{\1,\2\}} B_j (\dd,\ad,\bs,\bt) r_j
\end{equation*}
for functions $A_j,B_j$ related by the following symmetries:
\begin{align*}
A_{\al,\1}(\dd,\ad,\bs,\bt)=A_{\al,\2}(-\dd,-\ad,\bt,\bs)&=
B_{\be,\1}(\dd,-\ad,\bs,\bt)=B_{\be,\2}(-\dd,\ad,\bt,\bs)\\
B_{\al,\1}(\dd,\ad,\bs,\bt)=B_{\al,\2}(-\dd,-\ad,\bt,\bs)&=
A_{\be,\1}(\dd,-\ad,\bs,\bt)=A_{\be,\2}(-\dd,\ad,\bt,\bs)\\
A_{\al,\1}(\dd,\ad,\bs,\bt)=A_{\al,\1}(-\ad,-\dd,\bs^*,\bt^*),&\quad
B_{\al,\1}(\dd,\ad,\bs,\bt)=B_{\al,\1}(-\ad,-\dd,\bt^*,\bs^*),
\end{align*}
where the star denotes the (harmonic) inversion interchanging the diagonals
$l_\dd$ and $l_\ad$ and fixing the midpoints of the sides.  It thus suffices
to compute $A_{\al,\1}(\dd,\ad,\bs,\bt)$ and $B_{\al,\1}(\dd,\ad,\bs,\bt)$. To
parameterize the points $\bs,\bt$ on the edges: a convenient reference space
is the pencil of lines through the intersection $O$ of the diagonals; this is
a projective line with four harmonically separated marked points (the two
diagonals $l_\dd$ and $l_\ad$ and the two lines $l_\al$ and $l_\be$ joining
$O$ to intersection points of opposite sides). In the concrete description of
\S\ref{s:quad}, the diagonals are $x=\pm y$ and the other lines are $x=0$ and
$y=0$. We set $l_\dd=0=[1:0]$, $l_\ad=\infty=[0:1]$, so that we can use
positive homogeneous coordinates $\bs=[s_0:s_1],\bt=[t_0:t_1]$ on the edges;
this fixes $s_j$ and $t_j$ up to independent scales. Each $A, B$ is a
polynomial of bidegree $(3,3)$ in $\bs, \bt$.

We compute that $A_{\al,\1}(\dd,\ad,\bs,\bt)$ is given (up to normalization)
by
\begin{align*}
&\qquad\qquad\qquad 2\;\;(s_0+s_1)\;\;\bigl( (1+\ad) t_0 + (1+\dd) t_1\bigr)\;\;
\bigl( (1-\ad) s_0 t_0 - (1-\dd) s_1 t_1  \bigr)^2\\
&- (1-\dd) (1-\ad)\; \bigl( (1-\ad) s_0 + (1-\dd) s_1\bigr)\;(t_0+t_1)\;
\bigl( (1+\ad) s_0 t_0  - (1+\dd) s_1 t_1 \bigr)^2,
\end{align*}
whereas $B_{\al,\1}(\dd,\ad,\bs,\bt)$ is given (up to normalization) by
\begin{align*}
&(1-\dd)(1+\ad)(1+\dd-\ad+\dd\ad)\,                      s_0^3    t_0^3\\
+2&(1 + \dd\ad + 1+\dd-\ad+\dd\ad)\,                     s_0^2s_1 t_0^3\\
+2&(1-\dd)(1+\ad)\bigl(4 - (1-\dd^2) (1-\ad^2)\bigr)\,    s_0^3   t_0 t_1^2\\
+&(1-\dd)(1+\ad)
    \bigl((1+\dd)^2+(1-\ad)^2+(1+\dd\ad)(2+\dd-\ad)\bigr)\,s_0^3  t_0^2 t_1\\
+\bigl(&(1+\dd\ad)
    (10+\dd-\ad+(1+\dd)(1-\ad))+(1+\dd)^2+(1-\ad)^2\bigr)\,s_0^2s_1 t_0^2t_1\\
+4&(1 + \dd \ad)\,                                        s_0s_1^2 t_0^2t_1\quad
+\quad (40 - 12(1-\dd^2)(1-\ad^2) - 8 \dd\ad)\,            s_0^2s_1 t_0t_1^2\\
+\bigl(&(1+\dd\ad)
    (10-\dd+\ad+(1-\dd)(1+\ad))+(1-\dd)^2+(1+\ad)^2\bigr)\,s_0 s_1^2 t_0 t_1^2\\
+&(1+\dd)(1-\ad)
    \bigl((1-\dd)^2+(1+\ad)^2+(1+\dd\ad)(2-\dd+\ad)\bigr)\,s_0s_1^2   t_1^3\\
+2&(1+\dd)(1-\ad)\bigl(4 - (1-\dd^2) (1-\ad^2)\bigr)\,     s_0^2s_1   t_1^3\\
+2&(1 + \dd \ad + 1-\dd+\ad+\dd\ad)\,		           s_1^3 t_0 t_1^2\\
+&(1+\dd)(1-\ad)(1-\dd+\ad+\dd\ad)\,	                   s_1^3    t_1^3.
\end{align*}

The latter expression typifies the contribution to the Futaki invariant from a
side which does \emph{not} meet the crease. Only the first and last two
coefficients can be negative, and this can happen if and only if
$B_{\alpha,0}(\dd,\ad,0,0)$ (i.e., $1+\dd-\ad+\dd\ad$) or
$B_{\alpha,0}(\dd,\ad,1,1)$ (i.e., $1-\dd+\ad+\dd\ad$) is negative. This means
that the expression already contributes negatively to the Futaki invariant of
one of the diagonals, in which case the normals can be scaled to make the
quadrilateral intemperate.

In contrast, the expression for $A_{\alpha,0}$ typifies the contribution to
the Futaki invariant from a side which \emph{does} meet the crease. Here we
have found a surprising factorization which shows that the contribution can be
negative even when $(1+|\dd|)(1+|\ad|)<2$ (so the quadrilateral is temperate
for any choice of normals). We deduce the following.

\begin{prop}\label{unstable} Let $\Delta$ be a compact convex quadrilateral.
\begin{bulletlist}
\item If $\Delta$ is a parallelogram, then for any affine normals $\Lab$,
  $(\Delta,\Lab)$ is K-polystable.
\item If $\Delta$ is not a parallelogram, then there exist choices for the
  affine normals $\Lab$ such that $(\Delta,\Lab)$ is K-polystable as well as
  choices such that $(\Delta,\Lab)$ is K-unstable.
\end{bulletlist}
\end{prop}
\begin{proof} The stability results are straightforward~\cite{Eveline}, but the
instability results stated in~\cite{Eveline} are incorrect: by
Example~\ref{orthotoric}, equipoised rational Delzant quadrilaterals are
K-polystable.  However, if $\Delta$ is not a parallelogram, then either
$\dd\neq \ad$ or $\dd\neq -\ad$. In the former case, put $s_0=(1-\ad)t_1$ and
$s_1=(1-\dd)t_0$ in $f^{\al}_{\bs,\bt}$ so that $A_{\al,\1}$ is negative. Then
$\cF_{\Delta,\Lab}(f^{\al}_{\bs,\bt})$ can be made negative by taking
$r_{\al,\1}$ large relative to the other inverse normals.  When $\dd\neq -\ad$
a similar argument applies to $\cF_{\Delta,\Lab}(f^{\smash\be}_{\bs,\bt})$ and
$B_{\be,\1}$.
\end{proof}

\section{Link with CR and sasakian $5$-manifolds}\label{s:CR}

There are well known connections between symplectic and K\"ahler geometry in
dimensions four and contact, CR and sasakian geometry in dimension
five~\cite{BGbook,Lerman,Sparks}. In particular, quasiregular Sasaki--Einstein
$5$-manifolds have K\"ahler--Einstein orbifolds as quotients by the Reeb
vector field, and this provides one way of constructing them. As observed by
D.~Martelli and J.~Sparks~\cite{Martelli-Sparks,Sparks}, the Sasaki--Einstein
manifolds of J.~Gauntlett, D.~Martelli, J.~Sparks, D.~Waldram~\cite{GMSW} and
M.~Cvetic, H.~Lu, D.~Page, and C.~Pope~\cite{CLPP} have quotients which are
of Calabi type and orthotoric respectively.

The general ambitoric context does not provide further K\"ahler--Einstein
examples, but the extremal metrics may be used to continuous families of
extremal sasakian 5-manifolds, as well as Reeb directions which do not admit
transversal extremal metrics (cf.~\cite{BGS,Eveline,Eveline2}).

\subsection{Contact, CR and sasakian structures}

Recall that a contact manifold is an odd dimensional manifold $N$ with a
maximally non-integrable codimension one distribution $\cH\subset TN$, i.e.,
the Lie bracket $(X,Y)\mapsto [X,Y] \mod \cH$ defines a nondegenerate
$TN/\cH$-valued $2$-form $\Omega$ on $\cH$ called the \emph{Levi form}. We
assume that the line bundle $TN/\cH$ is oriented; positive sections $\eta$ of
the \emph{contact line bundle} $(TN/\cH)^*\subset T^*N$ are called
\emph{contact forms}. Such a contact form has pointwise kernel $\cH$ and
induces a unique vector field $K$ with $\eta(K)=1$ and $\cL_K\eta=0$, called
the \emph{Reeb vector field} of $\eta$.

An \emph{almost CR structure} is a complex structure $J$ on $\cH$ such that
the Levi form is $J$-invariant, and $(N,\cH,J)$ is said to be a \emph{CR
manifold of Sasaki type} if there is a contact form $\eta$ such that
\begin{equation*}
g=dr^2+r^2(\d\eta(\cdot,J\cdot)+\eta^2), \qquad\omega=d(r^2\eta)= 2r\,
dr\wedge\eta+r^2\,\d\eta
\end{equation*}
is a K\"ahler metric on the cone $N\times\R^+$. The corresponding metric
$g_\eta=\d\eta(\cdot,J\cdot)+\eta^2$ on $M$ is called compatible sasakian
metrics. On a CR manifold of Sasaki type the contact forms $\eta$ giving rise
to compatible sasakian metrics are those for which $(\cH,J,\d\eta|_\cH)$ is
invariant under the Reeb vector field and descends to a K\"ahler structure on
local quotients by $K$; it is called the \emph{transverse K\"ahler geometry}.
The sasakian structure is said to be \emph{quasiregular} if the quotient by
$K$ is an orbifold.

\subsection{CR structure associated to positive ambitoric metrics}

We observe here that for fixed $A(z)$ and $B(z)$, the ambitoric K\"ahler
metrics $(g_+,\omega_+)$ we obtain form a family of sasakian metrics
compatible with a fixed $5$-dimensional CR-structure: this is similar to the
well-known identification of Bochner-flat K\"ahler metrics with sasakian
structures compatible with the standard CR structure on an odd-dimensional
sphere~\cite{Webster}.

Suppose that $(M,g_\pm,J_\pm,\omega_\pm,\tor)$ be a regular ambitoric
$4$-orbifold. Then on the union $M^0$ of the generic orbits the coordinates
$(\bx,\by,\ang)$ provide a diffeomorphism of $M^0$ with $D^0\times
\tor/2\pi\Lambda$ for a domain $D^0$ in (an affine patch of) $\Proj(W)\times
\Proj(W)$. Recall that $\tor\cong S^2W^*/\spn{q}$ and the space of
hamiltonians $\torh_+$ is isomorphic to $S^2W^*$. By passing to the
universal cover of of $\tor/2\pi\Lambda$, or introducing a lattice
$\tilde\Lambda\subset \torh_+$ covering $\Lambda$, we can pull back the
K\"ahler structure along $\pi_q\colon N^0=D^0\times \torh_+\to D^0\times
\tor$. Then
\begin{equation*}
\pi_q^*\omega_+ =\d\eta_q, \quad\text{where}\quad
\eta_q = - \frac{\ip{\d\ang,\bx\otimes \by}}{q(\bx,\by)},
\end{equation*}
where $\d\ang$ is the tautological $\torh_+=S^2W^*$ valued 1-form on the
5-manifold.  The kernel $\cH\cong \pi_q^*TM^0$ of $\eta_q$, together with
$J_+$ defines a CR structure of sasakian type on $N^0$. The Reeb field of the
contact form $\eta_q$ is
\begin{equation*}
K_q= -\ip{q,X} =
-(q_0 \partial_{t_0} + q_1 \partial_{t_1} + q_2 \partial_{t_2}),
\end{equation*}
where $X\in \torh_+{}^{\!*}\otimes C^\infty(N^0,TN^0)$ is dual to $\d\ang$, and
the corresponding sasakian metric is $g_+ + {\eta_q}^2$. Keeping the CR
structure fixed, we now rescale the contact form and define
\begin{align*}
\eta&=\frac{q(\bx,\by)}{\kappa(\bx,\by)}\eta_q=
- \frac{\ip{\d\ang,\bx\otimes \by}}{\kappa(\bx,\by)}\quad\text{with}\\
\d\eta&= \frac{- \d\bx\wedge \ip{\d\ang,\by\otimes \by}
+ \d\by \wedge \ip{\d\ang,\bx\otimes \bx}}{\kappa(\bx,\by)}.
\end{align*}
The corresponding Reeb vector field is 
\begin{equation*}
K=-\ip{x\odot y,X}=
- \partial_{t_0} + (1/2) (x+y)\partial_{t_1} - xy \partial_{t_2}.
\end{equation*}
This does not preserve the CR structure and hence only defines a normal
contact metric, not a sasakian metric. To compute this, we need to find the
horizontal lift of $g_0 = q(x,y) g_+ /(x-y)$.  For this we observe that
$\ip{\d\ang,\{q,\by\otimes \by\}}$ agrees with $q(\bx,\by)
\ip{\d\ang,\by\otimes \by}/\kappa(\bx,\by)$ on $\cH$ (i.e., modulo $\eta$) and
the latter vanishes on $K$. Similarly, we replace $\ip{\d\ang,\{q,\bx\otimes
  \bx\}}$ by $q(\bx,\by) \ip{\d\ang,\bx\otimes
  \bx}/\kappa(\bx,\by)$. Introducing affine coordinates, we conclude that the
contact metric is
\begin{multline*}
\frac{\d x^2}{A(x)} + \frac{\d y^2}{B(y)}
+ A(x)\Bigl(\frac{y^2 dt_0 + 2y dt_1 + dt_2}{(x-y)^2}\Bigr)^2
+ B(y)\Bigl(\frac{x^2 dt_0 + 2x dt_1 + dt_2}{(x-y)^2}\Bigr)^2\\
+ \Bigl(\frac{x y dt_0 + (x+y) dt_1 +dt_2}{x-y} \Bigr)^2,
\end{multline*}
which is manifestly independent of $q$. Consequently this CR structure has a
family of compatible sasakian structures $\eta_q$ (with Reeb vector fields
$K_q$) for $q\in S^2W^*$. If $A(z)$ and $B(z)$ are quartics such that
$A(z)+B(z)=q_1(z)q_2(z)$ for orthogonal quadratic forms $q_1$ and $q_2$, then
both sasakian structures ($q=q_1$ and $q=q_2$) will be extremal.

\end{document}